\documentclass[11pt]{article}
\usepackage{amssymb}
\usepackage{amsmath}
\usepackage{amsfonts}
\usepackage{textcomp}
\usepackage{amsthm}
\usepackage{bookmark}
\usepackage{arydshln}
\usepackage{dsfont}
\usepackage{changes}
\usepackage{color}
\usepackage{verbatim}
\usepackage{cite}

\date{\today}

\newtheorem{Theorem}{Theorem}[section]
\newtheorem{Lemma}{Lemma}[section]
\newtheorem{Remark}{Remark}[section]

\newtheorem{Definition}{Definition}[section]

\numberwithin{equation}{section} \theoremstyle{plain}

\def\R{{\textbf{R}}}
\def\H{\mathcal{H}}

\def\dsum{\displaystyle\sum}

\def\f{\frac}
\def\R{\mathbb R}
\def\N{\mathbb N}

\def\e{\mathrm{e}}

\def\g{\gamma}

\def\al{\alpha}

\def\ds{\displaystyle}
\def\l{\left}
\def\r{\right}

\def\n{\nabla}
\title{Ground States of Attractive Fermi Schr\"{o}dinger Systems with Ring-Shaped Potentials}
\author{
 Yujin Guo\thanks{School of Mathematics and Statistics, and Key Laboratory of Nonlinear Analysis $\&$ Applications (Ministry of Education), Central China Normal University,  Wuhan 430079, P. R. China.  Y. J. Guo is partially supported by National Key R $\&$ D Program of China (Grant 2023YFA1010001), and NSF of China (Grants 12225106 and 12371113). Email: \texttt{yguo@ccnu.edu.cn}.},
\, Yan Li\thanks{School of Mathematics and Statistics, Central China Normal University,  Wuhan 430079,
	P. R. China.  Email: \texttt{yanlimath@mails.ccnu.edu.cn}.},
\, and\, Shuang Wu\thanks{School of Mathematics and Statistics, Central China Normal University,  Wuhan 430079,
	P. R. China. Email: \texttt{swu@mails.ccnu.edu.cn}.}}

\begin{document}
\maketitle

\begin{abstract}
As an application of the finite-rank Lieb-Thirring inequality established in [R. L. Frank, D. Gontier and M. Lewin, Comm. Math. Phys., 2021], we study ground states of mass-critical N-coupled Fermi nonlinear Schr\"{o}dinger systems with attractive interactions in $\R^3$, which are trapped in  ring-shaped potentials. For any given $N\in\N^+$, we prove that ground states exist if $0<a<a_N^*$, where $a$ denotes the strength of attractive interactions in the system, and $a_N^*$ is the best constant of a finite-rank Lieb-Thirring inequality. Moreover, for some $N\in\N^+$, we also prove the nonexistence of minimizers for the system as soon as $a\geq a_N^*$. Applying the energy estimates and the blow-up analysis, we further analyze the mass concentration behavior  of ground states for the system as $a\nearrow a_N^*$.
\end{abstract}

	
\bigskip
	
\section{Introduction}

There have been significant progresses (cf. \cite{BDZ,G-S}) on the experimental manipulations of cold atoms since the past few decades. A key feature of these experiments is the use of a trapping potential that confines the particles to a limited spatial region. The ability of varying the confining potential, which may also depend on the spatial directions, allows one (cf. \cite{KWW,SMS}) to control the effective spatial geometry of the particle systems. In particular, the trapped Fermi gases have garnered interest as the possible precursors of paired-fermion condensates at lower temperatures (cf. \cite{BR,Bru}) and have been studied experimentally in ultracold fermionic clouds (cf. \cite{De,Hol}). It turns out from physical experiments  (cf. \cite{PRA,PRL,PRX}) that the ring-shaped potentials play a crucial role in investigating the trapped fermionic atoms, as they represent the simplest multiple connected geometry for guiding coherent matter waves. These novel experimental advances have significantly contributed to the mathematical theories and numerical methods of trapped Fermi nonlinear Schr\"{o}dinger systems, see \cite{G-S,Naf,Dav,Dou,Bra,Lieb10} and the references therein.

When studying fermions, such as  electrons, neutrons or protons, it is natural to consider the system of orthonormal functions \cite{Lieb10}. Following the arguments of \cite{FGL21,GLN21,Lieb10}, the ground states  of $N$-spinless mass-critical Fermi nonlinear Schr\"{o}dinger systems  in $\R^3$ with attractive interactions and trapping potentials  can be described  as the minimizers of the  following constraint variational problem: for any given $N\in \mathbb{N}^+$,
    \begin{equation}\label{1.2}
		I_a(N):=\inf\Big\{E_a(u_1,\cdots,u_N):\, u_i\in \H,\ \langle u_i,u_j\rangle_{L^2(\R^3,\mathbb{C})}=\delta_{ij},\,1\leq i,\,j\leq N\Big\},\ a>0,
	\end{equation}
    where the energy functional $E_a(u_1,\cdots,u_N)$  is given by
	\begin{equation*}
    E_a(u_1,\cdots,u_N)=\sum_{i=1}^N\int_{\R^3}\Big(|\n u_i|^2+V(x)|u_i|^2\Big)dx-a\int_{\R^3}\Big(\sum_{i=1}^N|u_i|^2\Big)^{\f53}dx.
	\end{equation*}
Here $V(x)\geq 0$ denotes the trapping potential, $a>0$ represents the attractive strength of the quantum
particles in the system, and the Hilbert space $\H$ is defined as
	\begin{equation*}
		\H:=\Big\{u\in H^1(\R^3,\mathbb C):\int_{\R^3}V(x)|u(x)|^2dx<\infty\Big\}
	\end{equation*}
together with the norm
\begin{equation*}
\|u\|_\H:=\Big\{\int_{\R^3}\Big[|\n u|^2+(V(x)+1)|u|^2\Big]dx\Big\}^{\f 12}.
\end{equation*}
In particular, the notation $\|\|_{\H}$ depends on the type of $V(x)$.
	

Following \cite{FGL21,GLN21,Lewin11}, we now denote $\g$ to be a non-negative self-adjoint operator on $L^2(\R^3,\mathbb C)$ satisfying $\operatorname{Tr}(\g)=N$, so that $\g$ is compact. By the spectral theorem, the operator $\g$ can be further diagonalized as
\begin{equation*}  \gamma=\displaystyle\sum_{i=1}^\infty n_i|u_i\rangle\langle u_i|,\ \ \langle u_i,u_j\rangle_{L^2(\R^3,\mathbb C)}=\delta_{ij}\ \ \hbox{for}\ \ i,j=1,2,\cdots,
\end{equation*}
where $n_i\geq 0$ and
$\sum_{i=1}^\infty n_i=N$. Moreover, we have
	\begin{equation*}
		\g\varphi(x)=\sum_{i=1}^\infty n_i\langle u_i,\varphi\rangle_{L^2(\R^3,\mathbb C)}u_i(x),\ \  \forall\varphi(x)\in L^2(\R^3,\mathbb C),
	\end{equation*}
and the corresponding density of the operator $\g$ is defined by
\begin{equation}\label{density}
\rho_\g(x):=\g(x,x),
\end{equation}
where $\g(x,y)=\sum_{i=1}^\infty n_iu_i(x)\bar u_i(y)$ denotes the integral kernel of $\g$.  The kinetic energy of the operator $\g$ is then defined as
\begin{equation*}
		\operatorname{Tr}(-\Delta \g):=\sum_{i\geq1}\int_{\R^3}n_i|\n u_i(x)|^2dx.
\end{equation*}
By the similar arguments of \cite[Appendix A]{CG24} and \cite[Lemma 11]{GLN21}, one can then reduce equivalently the problem \eqref{1.2} to the following form:
\begin{equation}\label{infN}
\begin{split}
I_a(N)=\inf \Big\{&E_a(\gamma):\ \gamma=\displaystyle\sum_{i=1}^N|u_i\rangle\langle u_i|,\ \, u_i\in \H,\\&\qquad\qquad\langle u_i,u_j\rangle_{L^2(\R^3,\mathbb C)}=\delta_{ij},\ \,1\leq i,j \leq N \Big\},\ \,a>0,\ \,N\in\N^+,
\end{split}
\end{equation}
where the energy functional $E_a(\gamma)$ is given by
\begin{equation}\label{functional}
E_a(\gamma):=\operatorname{Tr}\big(-\Delta+V(x)\big)\gamma-a\int_{\R^3}\rho_\gamma^{\f53}dx.
\end{equation}

It turns out that the analysis of $I_a(N)$ depends on the following finite-rank Lieb-Thirring inequality established in \cite{FGL21}:
\begin{equation}\label{a*}
a^*_N:=\inf\Big\{\f{\|\gamma\|^{\f23}\operatorname{Tr}(-\Delta\gamma)}{\int_{\R^3}\rho_\gamma^{\f53}dx}:\,\g=\displaystyle\sum_{i=1}^{N}n_i|u_i\rangle\langle u_i|\neq0,\, u_i\in H^1(\R^3,\,\mathbb C),\,n_i\geq0\Big\}.
	\end{equation}
Here $\rho_\g=\sum_{i=1}^Nn_i|u_i|^2$ is defined as in \eqref{density}, and $\|\g\|$ denotes the norm of the operator $\g$. 	 It follows from \cite [Theorem 6]{FGL21} that for any $N\in\N^+$, the best constant $a_N^*$ is attained, and any optimizer $\g^{(N)}$ of $a_N^*$ can be written as
	\begin{equation}\label{gN}
		\gamma^{(N)}=\| \gamma ^{(N)}\| \sum_{i=1}^{R_N}|Q_i\rangle\langle Q_i|,\quad\langle Q_i,Q_j\rangle_{L^2(\R^3)}=\delta_{ij}\ \text{ for }\ 1\leq i,\,j\leq R_N,
	\end{equation}
where  $R_N\in[1,N]$ is a positive integer, and the orthonormal system $(Q_1,\cdots,Q_{R_N})$ solves the following fermionic nonlinear Schr\"{o}dinger system
	\begin{equation}\label{oH}
		\Big[-\Delta-\f53a_N^*\Big(\ds\sum_{j=1}^{R_N}Q_j^2\Big)^{\f23}\Big]Q_i=\hat\mu_iQ_i\ \text{ in }\ \R^3,\ \ i=1,\cdots,R_N.
	\end{equation}
Here $\hat\mu_1<\hat\mu_2\leq\cdots\leq\hat\mu_{R_N}<0$ are the  $R_N$ first negative eigenvalues (counted with multiplicity) of the operator
	\begin{equation*}
		 H_{\g^{(N)}}:=-\Delta-\f53a_N^*\Big(\ds\sum_{j=1}^{R_N}Q_j^2\Big)^{\f23}\ \text{ in }\ L^2(\R^3).
	\end{equation*}
Moreover, note from \cite[Theorem 6]{FGL21} that
	there exists an infinite sequence of integers $N_1=1< N_2=2< N_3<\cdots$  such that
	\begin{equation}\label{in}
		a_{N_{m}-1}^*>a_{N_m}^*,\quad m=2,3,4,\cdots,
	\end{equation}
and any optimizer $\g^{(N_m)}$ of $a_{N_m}^*$ is full-rank, $i.e.$, Rank$(\g^{(N_m)})=N_m$.

The application of  the finite-rank Lieb-Thirring inequality was studied recently in  \cite{GLN21}, where the authors mainly analyzed the mass-subcritical case of $I_a(N)$ by employing a similar version of (\ref{a*}).
Stimulated by this fact and the references \cite{GLN21,Lund,Lund1}, as an application of the finite-rank Lieb-Thirring  inequality established in \cite{FGL21}, {\it the main purpose of the present paper is to apply the minimization problem (\ref{a*}) for analyzing the mass-critical problem $I_a(N)$}. We also remark that one needs to impose a parameter $a>0$ in front of the nonlinear term for  the mass-critical problem $I_a(N)$, which was already pointed out in \cite[pp. 1208]{GLN21}.

Since the energy functional $E_a(\g)$ of \eqref{functional} is invariant under the unitary transformation, if $\gamma=\displaystyle\sum_{i=1}^N|u_i\rangle\langle u_i|$ is a  minimizer of $I_a(N)$,  by the variational argument, then we may assume that after some suitable transformation, the vector function $(u_1,\cdots, u_N)$ satisfies the following fermionic  nonlinear Schr\"odinger system
\begin{equation}\label{NLS}
		H_Vu_i:=	\l[-\Delta+V(x)-\f53a\Big(\sum_{i=1}^N|u_i|^2\Big)^{\f23}\r]u_i=\mu_i u_i\ \ \text{in }\ \R^3,\,a>0,\,i=1,\cdots,N,
	\end{equation}
where $\langle u_i,u_j\rangle_{L^2(\R^3)}=\delta_{ij}$ holds for $i,\,j=1,\cdots,N\in\N^+$,  and $\mu_i$ is a suitable Lagrange multiplier for $i=1,\cdots,N$. We further expect that for any $N\in \mathbb N^+$, the minimizers of $I_a(N)$ are connected with ground states of  \eqref{NLS}, in the following sense that

\begin{Definition}(Ground states)\label{ground}
An  $L^2$-orthonormal system $(u_1,\cdots,u_N)\in\H \times \H \cdots \times\H$  is called a ground state of \eqref{NLS}, if it solves \eqref{NLS}, where
$$\mu_1<\mu_2\leq\mu_3\leq\cdots\leq\mu_N$$
are the $N$ first eigenvalues (counted with multiplicity) of the operator $H_V$ in $L^2(\R^3)$.
	\end{Definition}

By employing the minimization problem (\ref{a*}), the first  result of this paper is to study the following existence and nonexistence of minimizers for  $I_a(N)$ defined by \eqref{infN}.
	
\begin{Theorem}\label{groundstate}
Suppose the potential $V(x)$ satisfies
\begin{equation*}
 0= \min_{x\in\R^3} V(x)\leq V(x) \in C(\R^3)\text{ \ and \ }\ds\lim_{|x|\to\infty}V(x)=+\infty,
\end{equation*}
and let $a_N^*>0$ be given by \eqref{a*}, where $N\in\N^+$ satisfies that $a_N^*$ admits a full-rank optimizer.
Then we have
\begin{enumerate}
\item[(1)] If $0<a<a_N^*$, then there exists  at least one minimizer $\g=\displaystyle\sum_{i=1}^{N}|u_i\rangle\langle u_i|$ of $I_a(N)$, where $\left(u_1,\dots,u_N\right)$  is a ground state of \eqref{NLS}.

\item[(2)] If $a\geq a_N^*$, then there is no minimizer of $I_a(N)$. 	Moreover, $\ds\lim_{a\nearrow a_N^*}I_{a}(N)=I_{a_N^*}(N)=0$, and $I_a(N)=-\infty$ for $a>a_N^*.$
\end{enumerate}
\end{Theorem}

\begin{Remark}
As for Theorem \ref{groundstate}, the restriction that $a_N^*$ admits a full-rank optimizer, which holds at least for $N=1, 2$ in view of \cite[Theorem 6]{FGL21}, is only used to ensure the nonexistence of Theorem \ref{groundstate} (2).
\end{Remark}

Since $\lim_{a\nearrow a_N^*}I_{a}(N)=I_{a_N^*}(N)=0$, the proof of Theorem \ref{groundstate} yields that $$\int_{\R^3}V(x)\rho_{\gamma_a}dx\to0=\inf_{x\in\R^3}V(x)\ \ \mbox{as}\ \ a\nearrow a_N^*.$$
Therefore, the limiting behavior of the minimizers for $I_a(N)$ as $a\nearrow a_N^*$ is influenced by the local behavior of $V(x)$ near its minimum points. When the trap $V(x)$ is a Coulomb potential, the mass concentration behavior of minimizers for $I_a(N)$ as $a\nearrow a_N^*$ was first studied in \cite{Chen1}, whose analysis depends strongly on the trap $V(x)$  having a finite number of minimum points. It is thus natural to ask what happens if $V(x)$ has infinitely many minimum points. Therefore, we next analyze mainly the limiting behavior of the minimizers for $I_a(N)$ for the case where the trap $V(x)$ is a ring-shaped potential, $i.e.,$ $V(x)$ has infinitely many minimum points. We remark that the ring-shaped potentials were considered  in recent physical experiments of Fermi gases, see \cite{rsp,PRA, PRL, PRX} and the references therein.

In view of above facts, we now consider the following ring-shaped potential in terms of the radial frequency $\omega_1>0$ and the vertical frequency $\omega_2>0$:
\begin{equation}\label{eV}
V(x)=\omega_1\left(r-A\right)^2+\omega_2x_3^2, \ \, A>0,
\end{equation}
where $x=(x_1,x_2,x_3)\in\R^3$ and $r=\sqrt{x_1^2+x_2^2}$. Applying the minimization problem (\ref{a*}), under the above assumption (\ref{eV}), we shall prove the following limiting behavior of the minimizers for $I_a(N)$ as $a\nearrow a_N^*$:
	
\begin{Theorem}\label{1.2T}
Let $V(x)\ge 0$ be given by \eqref{eV}, and suppose $N\in\N^+$ is chosen  so that any optimizer of $a_N^*$ is full-rank. Assume $\g_{{a}}=\sum_{i=1}^N| u^a_i\rangle\langle u^a_i|$ is a minimizer of $I_{a}(N)$, where $u^a_i$ satisfies \eqref{NLS} for $i=1,\cdots,N$. Then for any sequence $\{\g_{a_n}\}$ satisfying $a_n\nearrow a_N^*$ as $n\to\infty$, there exists a subsequence, still denoted by $\{\g_{a_n}\}$, of  $\{\g_{a_n}\}$ such that for  $i=1,\cdots,N$,
\begin{equation}\label{A-1}
\begin{aligned}
 w_i^{a_n}(x):&=(a_N^*-a_n)^{\f38}u_i^{a_n}\big((a_N^*-a_n)^{\f14} x+x_n\big)\\
&\to  w_i(x)\ \text{ strongly in}\ H^1(\R^3)\cap L^\infty(\R^3)\ \ \hbox{as}\ \ n\to\infty,
\end{aligned}
\end{equation}
where $\g:=\sum_{i=1}^N|   w_i\rangle\langle   w_i|$ is an optimizer of $a_N^*$ defined by \eqref{a*}, and the global maximum point $x_n=(p_n,z_n)\in\R^2\times\R$ of the density $\rho_{\g_{a_n}}:=\sum_{i=1}^N|u_i^{a_n}|^2$ satisfies
\begin{equation*}
\ds\lim_{n\to\infty}x_n=\ds\lim_{n\to\infty}(p_n,z_n)=(p_0,0)\in\R^2\times\R, \text{ \ and\ } |p_0|=A>0  \text{\ is as in\ } (\ref{eV}).
\end{equation*}
Moreover, there exist constants $C_0$ and $C_1$ such that $x_n=(p_n,z_n)$ satisfies
\begin{equation}\label{C_0}
\lim_{n\to\infty}\f{|p_n|-|p_0|}{(a_N^*-a_n)^{\f14}}=C_0\ \text{ and }\ \lim_{n\to\infty}\f{z_n }{(a_N^*-a_n)^{\f14}}=C_1.
\end{equation}
Furthermore, the energy $I_{a_n}(N)$ satisfies
\begin{equation}\label{1-1}
\begin{split}
\lim_{n\to\infty}\frac{I_{a_n}(N)}{(a_N^*-a_n)^{\f12}}&=\int_{\R^3}\rho_{\g}^{\f 53}dx+\int_{\R^3}\Big[\omega_1\Big(\f{(x_1,x_2)\cdot p_0}{A}+C_0\Big)^2\\
& \,\qquad\qquad\qquad\qquad +\omega_2(x_3+C_1)^2\Big]\rho_{\g}dx,
\end{split}
\end{equation}
where $\rho_\g:=\sum_{i=1}^N|w_i|^2$, and $A>0$,  $\omega_1>0$ and $\omega_2>0$ are as in \eqref{eV}.

	\end{Theorem}
	

\begin{Remark}
The $L^\infty$-convergence of \eqref{A-1} implies that
$$\g_{a_n}(x,y)\approx (a_N^*-a_n)^{-\f34}\g\Big(\f{x-x_0}{(a_N^*-a_n)^{\f14}}, \f{y-x_0}{(a_N^*-a_n)^{\f14}}\Big)\ \ \hbox{as}\ \ a_n\nearrow a_N^*,$$
where $\g(x,y)=\sum_{i=1}^Nw_i(x)w_i(y)$ is the integral kernel of the optimizer  $\g$ for $a_N^*$, and $x_0:=(p_0,0)\in\R^2\times \R$ is a global minimum point of the ring-shaped potential $V(x)$. This further yields that the mass of the minimizers for $I_{a_n}(N)$ concentrates at a global minimum point of $V(x)$ as $a_n\nearrow a_N^*.$
\end{Remark}

In the following, we sketch the proof of Theorem \ref{1.2T} by three steps. The first step of proving Theorem \ref{1.2T} is to establish the delicate estimates of the energy $I_a(N)$ as $a\nearrow a_N^*$. Once the ring-shaped potential \eqref{eV} is considered, the estimates of $I_a(N)$ as $a\nearrow a_N^*$ are however involved with the extra  difficulty. Actually, if the potential $V(x)$ has finite minimum points (such as the Coulomb potential), then it is enough (cf. \cite{CG24,Chen1}) to consider the local region containing a single minimum point of $V(x)$ to derive a refined estimate of the potential energy $\int_{\R^3}V(x)\rho_{\gamma_a} dx$. It further yields the refined energy estimates of the system. However, this idea does not work well for the case of ring-shaped potentials. To overcome this difficulty,  we first establish in Lemma \ref{ree}  the following rough  estimates of $I_a(N)$:
\begin{equation}\label{reee}
C_1(a_N^*-a)^{\f 35}\leq I_a(N)\leq C_2(a_N^*-a)^{\f 12}\ \ \text{as }\ a\nearrow a_N^*,
\end{equation}
where $C_1$ and $C_2$ are positive constants. Based on \eqref{reee}, the optimal lower estimate of $I_a(N)$  as $a\nearrow a_N^*$ is further derived in Theorem \ref{ee}.

As the second step of proving Theorem \ref{1.2T}, we shall follow the first step to derive the $H^1$-convergence of \eqref{A-1}. Towards this purpose, we need the following finite-rank Lieb-Thirring inequality
\begin{equation}\label{frlt}
		L_N^*{\int_{\R^3} A^{\f52}(x)dx}\geq{\sum_{j=1}^N\big|\lambda_j\big(-\Delta-A(x)\big)\big|},\,\  0\leq A(x)\in L^{\f52}(\R^3)\backslash\{0\},
\end{equation}
where the best constant $L_N^*\in(0,+\infty)$ can be achieved (cf. \cite[Corollary 2]{FGL24}). Here $\lambda_i\big(-\Delta-A(x)\big)\leq 0$ denotes the $i$-th  min-max level of $H_A:=-\Delta-A(x)$ in $L^2(\R^3)$, which equals to the $i$-th negative eigenvalue (counted with multiplicity) of $H_A$ in $L^2(\R^3)$ if it exists, and vanishes otherwise. Applying Theorem \ref{ee} and  \eqref{frlt}, we shall obtain the $L^{\f53}$-convergence of the density $\rho_{\widetilde\g_{a_n}}:=\sum_{i=1}^N|w_i^{a_n}|^2$, where $w_i^{a_n}$ is as in \eqref{A-1}. Following this convergence, the $H^1$-convergence of \eqref{A-1} is then derived in Lemma \ref{4.1L}.

The third step of proving Theorem \ref{1.2T} is to  derive mainly the $L^\infty$-uniform convergence of \eqref{A-1}, which requires that $\mu_i^{a_n}<0$ holds for all $ i=1,\cdots,N$. Here $\mu_i^{a_n}$ is the  $i$-th eigenvalue    of  the operator $H_V:=-\Delta+V(x)-\frac 53 a_n\rho^{\frac 23}_{\gamma_{a_n}}$ in $L^2(\R^3)$ for $ i=1,\cdots,N$,  and $\gamma_{a_n}$ is a minimizer of $I_{a_n}(N)$. When $V(x)$ is a ring-shaped potential,  the operator $H_V$ has no essential spectrum, due to the fact that $0\leq V(x)\to\infty$ as $|x|\to\infty$. Hence,  the existing methods of \cite{CG24,Chen1,GLN21} cannot yield $\mu_i^{a_n}<0$ as $a_n\nearrow a_N^*$, where  $ i=1,\cdots,N$.  To overcome this difficulty,  we shall use the energy estimates of Theorem 3.1 to derive that the sequence $\{\sum_{i=1}^N(a_N^*-a_n)^{\frac12}\mu_i^{a_n}\}$ is bounded uniformly from above, and  the sequence $\{(a_N^*-a_n)^{\frac12}\mu_i^{a_n}\}$ is bounded uniformly from below for $i=1,\cdots,N$. By analyzing the properties of optimizers for $a_N^*$, this further helps us to derive that $\mu_i^{a_n}<0$ holds  for all $ i=1,\cdots,N$ as $a_n\nearrow a_N^*$. The $L^\infty$-uniform convergence of \eqref{A-1} is then obtained in Lemma \ref{lem4.2} by applying the exponential decay of minimizers.

This paper is organized as follows. In Section \ref{2}, we shall prove Theorem \ref{groundstate} on the existence and nonexistence of minimizers for $I_a(N)$. Section \ref{3} is mainly concerned with the energy estimates of $I_a(N)$ as $a\nearrow a_N^*$, based on which the proof of Theorem \ref{1.2T} is finally completed in Section \ref{4}.

\section{Existence of Minimizers for $I_a(N)$}\label{2}
In this section, we mainly establish Theorem \ref{groundstate} on the existence   and nonexistence of minimizers for $I_a(N)$, where $N\in\N^+$ is chosen such that $a_N^*$ defined by \eqref{a*} admits a full-rank optimizer.

Towards this purpose, we  first introduce the following compactness lemma (cf. \cite{Adams}, \cite[Theorem XIII.67]{Sim4}).

	\begin{Lemma}\label{ce}
		Suppose $0\leq V(x)\in L^\infty_{loc}(\R^3)$  satisfies $\lim_{|x|\to\infty}V(x)=\infty$. Then the embedding $\H\hookrightarrow L^q(\R^3)$ is compact for $2\leq q<6$.
	\end{Lemma}
We next use  Lemma \ref{ce} and \eqref{a*} to complete the proof of Theorem \ref{groundstate}.

\vspace{5pt}

\noindent{\bf Proof of Theorem 1.1.}
(1). For any fixed $N\in\N^+$, let $\gamma=\sum_{i=1}^N|u_i\rangle\langle u_i|$ be an operator satisfying $u_i\in\mathcal{H}$ and $\langle u_i,u_j\rangle_{L^2(\R^3)}=\delta_{ij}$ for $i,j=1,\cdots,N$. Since $V(x)\geq 0$,  we obtain from \eqref{functional} and \eqref{a*} that for  $0<a<a_N^*$,
\begin{equation}\label{bounded}
\begin{split}
E_a(\gamma)&=\operatorname{Tr}\big(-\Delta+V(x)\big)\gamma-a\int_{\R^3}\rho_\gamma^{\f53}dx\\
&\geq \Big(1-\f{a}{a^*_N}\Big)\operatorname{Tr}\left(-\Delta\g\right)+\operatorname{Tr}\big(V(x)\gamma\big)\geq0,
\end{split}
\end{equation}
which implies that $I_a(N)$ is bounded from below.

Let $\{\g_n\}$ be a minimizing sequence of $I_a(N)$, $i.e.$, $\g_n=\sum_{i=1}^N|u_i^{n}\rangle\langle u_i^{n}|$ satisfies $\lim_{n\to\infty}E_a(\g_n)=I_a(N)$  and $\langle u_i^{n},u_j^{n}\rangle_{L^2(\R^3)}=\delta_{ij}$ for $i,\,j=1,\,\cdots,\,N.$ Since \begin{equation*}
\operatorname{Tr}\left(-\Delta\g_n\right)=\sum_{i=1}^N\int_{\R^3}|\nabla u_i^{n}|^2dx\ \,\ \hbox{and}\ \ \operatorname{Tr}\big(V(x)\gamma\big)=\int_{\R^3}V(x)\Big(\sum_{i=1}^N|u_i^{n}|^2\Big)dx,
\end{equation*}
we deduce from  \eqref{bounded} that $\{u_i^{n}\}_n$ is bounded uniformly in $\H$ for $i=1,\cdots,N$. It then follows from Lemma \ref{ce} that there exist  a subsequence, still denoted by $\{u_i^{n}\}$, of $\{u_i^{n}\}$ and a function  $u_i\in\H$ such that for $i=1,\cdots,N$,
\begin{equation*}
u_i^{n}\rightharpoonup u_i \ \text{ weakly in } \ \H,\quad	u_i^{n}\to u_i\  \text{ strongly in }\  L^q(\R^3)\ \ \hbox{as}\ \ n\to\infty,\ \,2\leq q<6.
\end{equation*}
We thus conclude that
\begin{equation}\label{o}
\l<u_i,u_j\r>_{L^2(\R^3)}=\delta_{ij},\ \,i,\,j=1,\,\cdots,\,N,
\end{equation}
and
\begin{equation}\label{nc}
			\rho_{\g_n}:=\sum_{i=1}^N|u_i^{n}|^2\to\rho_\g:=\sum_{i=1}^N|u_i|^2\ \text{ strongly in }\ L^r(\R^3)\ \ \hbox{as}\ \ n\to\infty,\ \,1\leq r<3,
		\end{equation}
where $\g:=\sum_{i=1}^N|u_i\rangle\langle u_i|$. By the  weakly lower semi-continuity,	we deduce from \eqref{o} and \eqref{nc} that
\begin{equation*}
\begin{split}
I_a(N)&=\liminf_{n\to\infty}E_a(\g_n)\\
&\geq \operatorname{Tr}\left(-\Delta\g\right)+\int_{\R^3}V(x)\rho_\g dx-a\int_{\R^3}\rho_\gamma^{\f53}dx\\
&=E_a(\g)\geq I_a(N),
\end{split}
\end{equation*}
which implies that $\g$ is a minimizer of $I_a(N)$. Up to some unitary diagonalization, we assume that $u_i$ satisfies the following system
\begin{equation}\label{lm}
-\Delta u_i+V(x)u_i-\f53 a\rho_{\g}^{\f 23}u_i=\mu_i u_i\ \ \hbox{in}\ \ \R^3,\ \ i=1,\cdots,N,
\end{equation}
where $\mu_i\in\R$ is a suitable Lagrange multiplier. Without loss of generality, we may assume that $\mu_1\leq \mu_2\leq\cdots\leq\mu_N$.

In the following, we  prove that $(u_1,\dots,u_N)$ is a ground state of \eqref{NLS}, i.e.,    $\mu_1<\mu_2\leq\dots\leq\mu_N$ are  the $N$ first eigenvalues of the  operator
\begin{equation}\label{2-4}
H_V:=-\Delta +V(x)-\f{5a}3\rho_\g^{\f23}\ \, \hbox{in}\ \, L^2(\R^3),
\end{equation}
and $u_i$ is the associated eigenfunction of $\mu_i$ for $i=1,\cdots,N$. Actually, we first claim that the minimizer  $\g$ of $I_a(N)$ is also an optimizer of the following problem
\begin{equation}\label{HV}
\inf_{\g'\in\mathcal{K}_N}\,\operatorname{Tr}H_V(\g'),
\end{equation}
where $H_V$ is defined by \eqref{2-4}, and
\begin{equation*}
\mathcal{K}_N:=\big\{\g\in\mathcal{B}\big(L^2(\R^3,\R)\big):\  0\leq\g=\g^*\leq1,\,\operatorname{Tr}(\g)=N,\,\operatorname{Tr}(-\Delta \g)<\infty\big\}.
\end{equation*}
Here $\mathcal B\big(L^2(\R^3,\R)\big)$ denotes the set of bounded linear operators on $L^2(\R^3,\R)$. To prove (\ref{HV}), we note that for any $\g'\in\mathcal{K}_N$,
\begin{equation}\label{infh}
\begin{split}			 E_a(\g')&=E_a(\g)+\operatorname{Tr}H_V(\g'-\g)-a\int_{\R^3}\Big[\rho_{\g'}^{\f53}-\rho_{\g}^{\f53}-\f53\rho_{\g}^{\f23}(\rho_{\g'}-\rho_{\g})\Big]dx\\
&\leq E_a(\g)+\operatorname{Tr}H_V(\g'-\g),
\end{split}
\end{equation}
where the convexity of $x\mapsto x^{\f53}$ is used in the last inequality. The similar arguments of \cite[Lemma 11]{GLN21} yield that the problem \eqref{infN} is equivalent to the following form:
\begin{equation*}
  I_a(N)=\inf_{\g\in\mathcal{K}_N} E_a(\g),
\end{equation*}
where $E_a(\g)$ is defined by \eqref{functional}. Since $\g$ is a minimizer of $I_a(N)$,  we derive from \eqref{infh} that
\begin{equation*}
\operatorname{Tr}H_V(\g)	\leq\operatorname{Tr}H_V(\g'),\ \ \forall\g'\in\mathcal{K}_N,
\end{equation*}
Since $\g\in\mathcal{K}_N $, we deduce that $\g$ is an optimizer of \eqref{HV}, and hence the  above claim holds true.

We next claim that
\begin{equation}\label{continuous}
u_i\in C(\R^3)\ \text{ and }\lim_{|x|\to\infty}u_i (x)=0,\ \ i=1,\cdots,N.
\end{equation}
Indeed, applying  Kato's inequality (cf. \cite[Theorem X.27]{Sim2}), it yields from  \eqref{lm} that
\begin{equation*}
\left(-\Delta -\f53 a\rho_{\g}^{\f 23}\right)|u_i|\leq \mu_i |u_i|\ \,\ \mbox{in} \ \, \R^3,\ \,i=1,\cdots,N,
\end{equation*}
due to the fact that  $V(x)\geq0$. Since $u_i\in H^1(\R^3)$, by Sobolev's embedding theorem it gives that $u_i\in L^q(\R^3)$ for $2\leq q\leq 6$, which further implies that $\rho_\g\in L^r(\R^3)$  holds for $1\leq r\leq3$.
Since $\|\f53 a\rho_{\g}^{\f 23}\|_{L^2\big(B_2(y)\big)}\leq \f 53 a\|\rho_{\g}\|^{\f{2}{3}}_{L^{\f43}(\R^3)}$ holds for any $y\in\R^3$, by De Giorgi-Nash-Moser theory (cf. \cite[Theorem 4.1]{hl}), one can deduce    that for any $y\in\R^3$,
\begin{equation}\label{ub}
\|u_i\|_{L^\infty(B_1(y))}\leq C\|u_i\|_{L^2(B_2(y))},\ \,i=1,\cdots,N,
\end{equation}
where $C>0$ depends on $a,\,\mu_i$ and $\|\rho_{\g}\|_{L^{\f43}(\R^3)}$. It then follows from \eqref{ub} that  $u_i\in L^\infty(\R^3)$ and $\lim_{|x|\to\infty}|u_i(x)|=0$ for any fixed $a>0$. Since $V(x)\in L_{loc}^\infty(\R^3)$, it yields that $\big(V(x)-\f53 a\rho_{\g}^{\f 23}\big)u_i\in L^2_{loc}(\R^3).$   Applying $L^p$-theory (cf. \cite[Theorem 8.8]{gt}), we thus get from \eqref{lm} that $u_i(x)\in W^{2,2}_{loc}(\R^3)$.  By Sobolev's embedding theorem, we hence derive that the claim \eqref{continuous} holds true.

Since $V(x)\in L_{loc}^\infty(\R^3)$, it follows  from \eqref{continuous} that
\begin{equation*}
V(x)-\f{5a}3
\rho_\g ^{\f23}\in L_{loc}^2(\R^3)\ \, \text{ and }\ \lim_{|x|\to\infty}\Big(V(x)-\f{5a}3
			\rho_\g^{\f23}\Big)=\infty.
\end{equation*}
Note  from \cite[Theorems \uppercase\expandafter{\romannumeral10}\uppercase\expandafter{\romannumeral3}.47 and  \uppercase\expandafter{\romannumeral10}\uppercase\expandafter{\romannumeral3}.67]{Sim4} that $H_V$ has purely discrete spectrum, and its first eigenvalue  is   simple, whose eigenfunction is strictly positive. In particular, $H_V$ admits at least $N$ eigenvalues, counted with multiplicity. Let $\lambda_1<\lambda_2\leq\cdots\leq\lambda_N$ be the $N$ first  eigenvalues of $H_V$ in $L^2(\R^3)$, and suppose $v_i$ is the  eigenfunction associated to $\lambda_i$, where $\langle v_i,v_j\rangle_{L^2(\R^3)}=\delta_{ij}$ holds for $i,j=1,\cdots,N$.
Since $\g$ is a minimizer of \eqref{HV}, we have for $\tilde{\g}:=\sum_{i=1}^N|v_i\rangle\langle v_i|$,
\begin{equation*} \sum_{i=1}^N\lambda_i\leq\sum_{i=1}^N\mu_i=\operatorname{Tr}H_V(\g)=\inf_{\g'\in\mathcal{K}_N}\,\operatorname{Tr}H_V(\g')\leq \operatorname{Tr}H_V(\tilde\g)=\sum_{i=1}^N\lambda_i,
		\end{equation*}
which gives that
\begin{equation*}
\sum_{i=1}^N\lambda_i=\sum_{i=1}^N\mu_i,
\end{equation*}
and hence $\mu_1<\mu_2\leq\dots\leq \mu_N$ are the $N$ first eigenvalues of the operator $H_V$ in $L^2(\R^3)$. Therefore, $(u_1,\cdots,u_N)$ is a ground state of $I_a(N)$.

(2). Let
\begin{equation}\label{2-1}
			\gamma^{(N)}= \sum_{i=1}^N|Q_i\rangle\langle Q_i|,\ \ \left<Q_i,Q_j\right>_{L^2(\R^3)}=\delta_{ij}\ \text{ for }\ i,\,j=1,\cdots,N,
		\end{equation}
be an optimizer of  $a_N^*$ defined by \eqref{a*}.   Recall from \cite[Proposition 11]{FGL21} that
\begin{equation}\label{Q}
Q_i\in C^\infty(\R^3), \text{ and  }\ |Q_i(x)|,\,|\nabla Q_i(x)|=O\Big(\f{\e^{-\sqrt{|\hat\mu_i|}|x|}}{|x|}\Big)\ \text{ as }\ |x|\to\infty,\,\ i=1,\cdots,N,
\end{equation}
where $\hat\mu_i<0$ denotes the $i$-th eigenvalue (counted with multiplicity) of the operator   $-\Delta-\f53a_N^*\rho_{\g^{(N)}}^{\f23}$ in $L^2(\R^3)$.
Choose a cut-off function $\varphi\in C_c^\infty(\R^3,[0,1])$ such that $\varphi(x)\equiv1$ for $|x|<1$ and $\varphi(x)\equiv0$ for $|x|>2$. For any $y_0\in\R^3$ and $\tau>0$,  define
\begin{equation*}
Q_i^\tau(x)=A_i^\tau\tau^{\f32}\varphi(x-y_0)Q_i\big(\tau\l(x-y_0\r)\big),\ \ i=1,\cdots,N,
\end{equation*}
where $A_i^\tau>0$ is chosen such that $\int_{\R^3}|Q_i^\tau(x)|^2dx=1$. By the exponential decay of $Q_i$ in \eqref{Q}, it gives that
\begin{equation} \label{At}
\begin{split}
1\leq (A_i^\tau)^2&=\f{1}{\tau^3\int_{\R^3}\varphi^2(x-y_0)Q_i^2\l(\tau\l(x-y_0\r)\r)dx}\\
&=\f{1}{1-\int_{\R^3}\l(1-\varphi^2(\f{x}{\tau })\r)Q^2_i\l(x\r)dx}\\
&\leq \f{1}{1-\int_{|x|>\tau}Q^2_i\l(x\r)dx}\\
&=1+O(\tau^{-\infty})\,\ \hbox{as}\ \ \tau\to\infty,\ \ i=1,\cdots, N,
\end{split}
\end{equation}
where and below the notation $f(t)=O(t^{-\infty})$ denotes that
$\displaystyle\lim_{t\to\infty}t^s|f(t)|=0$ for any $s>0$.
It then follows from  \eqref{2-1}--\eqref{At} that
\begin{equation}\label{at}
\begin{split}
a_\tau:=\displaystyle{\max_{ i\neq j}}\big|\left<Q_i^\tau, Q_j^\tau\right>_{L^2(\R^3)}\big|=O(\tau^{-\infty})\ \text{ as }\ \tau\to\infty.
\end{split}
\end{equation}
We thus deduce that the Gram matrix
\begin{equation}\label{gram}
G_\tau:= \left(
\begin{array}{c}
Q_1^\tau \\
\vdots \\
Q_N^\tau \\
\end{array}
\right)
\left(
\begin{array}{ccc}
Q_1^\tau & \cdots & Q_N^\tau \\
\end{array}
\right)=\left(
\begin{array}{ccc}
1 & \cdots & \langle Q_1^\tau,Q_N^\tau\rangle \\
\vdots & \ddots & \vdots \\
\langle Q_N^\tau,Q_1^\tau\rangle& \cdots & 1 \\
\end{array}
\right)
\end{equation}
is positive definite for sufficiently large $\tau>0$.

Define for sufficiently large  $\tau>0$,
\begin{equation}\label{nQ}	
\big(\widetilde{Q}_1^\tau \cdots \widetilde{Q}_N^\tau\big)	:=\big(Q_1^\tau\cdots Q_N^\tau\big) G_\tau^{-\f 12}.
\end{equation}
It then follows from \eqref{gram}   that for sufficiently large  $\tau>0$,
\begin{equation*}
\l<\widetilde{Q}_i^\tau,\widetilde{Q}_j^\tau\r>_{L^2(\R^3)}=\delta_{ij},\ \,1\leq i,j\leq N.
\end{equation*}
Since $(1+x)^{-\f12}=1-\f12 x+O(x^2)$ as $x\to0$, one can deduce from \eqref{gram} and \eqref{nQ} that
\begin{equation}\label{taylor}
\big(\widetilde{Q}_1^\tau \cdots \widetilde{Q}_N^\tau\big)	=\big(Q_1^\tau\cdots Q_N^\tau\big)+O(a_\tau)\ \ \hbox{as}\ \ \tau\to\infty.
\end{equation}
Denoting \begin{equation*}
\widetilde{\g}_\tau^{(N)}:= \sum_{i=1}^{N}|\widetilde Q_i^\tau\rangle\langle \widetilde Q_i^\tau|,
\end{equation*}
one can derive from \eqref{Q}--\eqref{at} and \eqref{taylor} that
\begin{equation}\label{trd}
\begin{split}
E_a(\widetilde{\g}_\tau^{(N)})&=\operatorname{Tr}(-\Delta \widetilde{\g}_\tau^{(N)})+\operatorname{Tr}\big(V(x){\widetilde{\g}_\tau^ {(N)}}\big)-a\int_{\R^3}\rho_{\widetilde{\g}_\tau^{(N)}}^{\f53}dx\\
&=\dsum_{i=1}^N\int_{\R^3}|\nabla \widetilde{Q}^\tau_i|^2dx+\dsum_{i=1}^N\int_{\R^3}V(x)|\widetilde{Q}^\tau_i|^2dx-a\int_{\R^3}\Big(\dsum_{i=1}^N|\widetilde{Q}^\tau_i|^2\Big)^{\f53}dx\\
&=\tau^2	\Big[\operatorname{Tr}\big(-\Delta{\g}^{(N)}\big)-a\int_{\R^3}\rho_{\gamma^{(N)}}^{\f53}dx\Big]\\
&\quad+\int_{\R^3}V\Big(\f{x}{\tau}+y_0\Big)\varphi^2\l(\f{x}{\tau}\r)\rho_{\g^{(N)}}dx+O(\tau^{-\infty})\\
&=\tau^2\l(a_N^*-a\r)\int_{\R^3}\rho_{\gamma^{(N)}}^{\f53}dx+\int_{\R^3}V\Big(\f{x}{\tau}+y_0\Big)\varphi^2\l(\f{x}{\tau}\r)\rho_{\g^{(N)}}dx\\
&\quad+O(\tau^{-\infty})\ \text{ as }\ \tau\to\infty,\\
\end{split}
\end{equation}
where the last equality follows from the fact that $\g^{(N)}$ is an optimizer of $a_N^*$ satisfying $\|\g^{(N)}\|=1$.

On the other hand, since the function $x\mapsto V(\f{x}{\tau}+y_0)\varphi^2(\f{x}{\tau})$ is bounded  uniformly and has compact support, we obtain from  Lebesgue's dominated convergence that
\begin{equation}\label{V(y0)}
\lim_{\tau\to\infty}\int_{\R^3}V\Big(\f{x}{\tau}+y_0\Big)\varphi^2\l(\f{x}{\tau}\r)\rho_{\g^{(N)}}dx=NV(y_0),
\end{equation}
where $V\in C(\R^3)$ is also used.
It then follows from \eqref{trd} and \eqref{V(y0)} that for $a>a_N^*$,
\begin{equation*}
			I_a(N)\leq\lim_{\tau\to\infty}E_a(\widetilde{\g}_\tau^{(N)})=-\infty,
\end{equation*}
which implies that there is no minimizer of  $I_a(N)$ for  $a>a_N^*$.

For the case $a=a_N^*$, we derive from \eqref{trd} and \eqref{V(y0)} that $I_{a_N^*}(N)\leq N V(y_0)$ for each $y_0\in\R^3$. Taking the infimum over $y_0$, it then yields that $I_{a_N^*}(N)\leq0$. Thus, one can  deduce from \eqref{bounded} that $I_{a_N^*}(N)=0$. On the contrary, suppose that there exists a minimizer $\gamma_N= \sum_{i=1}^N|u_i\rangle\langle u_i|$ of $I_{a_N^*}(N)$, where $\left<u_i,u_j\right>_{L^2(\R^3)}=\delta_{ij}$ for $i,\,j=1,\cdots,N$.
We then obtain that
\begin{equation}\label{2-2}
			\int_{\R^3}V(x)\rho_{\g_N}dx=\inf_{x\in\R^3}V(x)=0,
\end{equation}
and
\begin{equation}\label{2-3}
\operatorname{Tr}(-\Delta\gamma_N)=a_N^*{\int_{\R^3}\rho_{\gamma_N}^{\f53}dx}.
\end{equation}
Note from \eqref{2-2} that $\rho_{\gamma_N}$ has compact support, due to the fact $\lim_{|x|\to\infty}V(x)=\infty$. However, it yields from \eqref{2-3} that $\gamma_N$ is an optimizer of $a_N^*$, and hence $u_1$ is the first eigenfunction of the operator $-\Delta-\frac{5}{3}a_N^*\rho_{\g_N}^{\frac{2}{3}}$ in $L^2(\R^3)$. It then follows from \cite[Theorem 11.8]{Lieb01} that $u_1^2>0$ in $\R^3$, which implies that $\rho_{\g_N}=\sum_{i=1}^N|u_i|^2>0$ in $\R^3$, a contradiction. Therefore, there is no minimizer for the case  $a=a_N^*$.

The above analysis yields that $I_a(N)=-\infty$ holds for $a>a_N^*$ and $I_{a_N^*}(N)=0$. Also, it follows easily from \eqref{trd} and \eqref{V(y0)} that $\lim_{a\nearrow a_N^*}I_{a}(N)=0$. This completes the proof of  Theorem \ref{groundstate}.\qed

\section{Energy Estimates of $I_a(N)$ as $a\nearrow a_N^*$}\label{3}

Under the assumption that $V(x)$ satisfies \eqref{eV}, the proof of Theorem \ref{groundstate} gives that $I_a(N)\to 0$ as $a\nearrow a_N^*$. The main purpose of this section is to establish the following   refined   estimates of  $I_a(N)$ as $a\nearrow a_N^*$.

\begin{Theorem}\label{ee}
Suppose $V(x)\ge 0$ is given by \eqref{eV}, and let $N\in\mathbb{N}^+$ be chosen such that $a_N^*$ admits a full-rank optimizer. Then there exist two positive constants $C_1$ and $C_2$, independent of $a>0$, such that
\begin{equation}\label{3.1}
C_1(a_N^*-a)^{\f 12}\leq I_a(N)\leq C_2(a_N^*-a)^{\f 12}\ \  \text{as}\ \ a\nearrow a_N^*.
\end{equation}
\end{Theorem}

In order to prove Theorem \ref{ee}, we first derive the following rough estimates of the energy $I_a(N)$ as $a\nearrow a_N^*$.

\begin{Lemma}\label{ree}
Suppose $V(x)\ge 0$ is given by \eqref{eV}, and let $N\in\mathbb{N}^+$ be chosen such that $a_N^*$ admits a full-rank optimizer.  Then there exist two positive constants $C_3$ and $C_4$, independent of $a>0$, such that
\begin{equation}\label{3.2}
C_3(a_N^*-a)^{\f 35}\leq I_a(N)\leq C_4(a_N^*-a)^{\f 12}\ \ \text{as}\ \ a\nearrow a_N^*.
\end{equation}
\end{Lemma}

\noindent{\bf Proof.}
Let $\g=\sum_{i=1}^N|u_i\rangle\langle u_i|$ be an operator satisfying  $u_i\in\mathcal{H}$ and $\l<u_i,u_j\r>_{L^2(\R^3)}=\delta_{ij}$ for $i,\,j=1,\cdots,N$. By Young's inequality, we obtain from \eqref{a*}  that for any  $\beta>0$ and $0<a<a_N^*$,
\begin{equation}\label{lowerea}
\begin{split}
E_a(\gamma)&\geq \int_{\R^3}V(x)\rho_\gamma dx+(a_N^*-a)\int_{\R^3}\rho_\gamma^{\f53}dx\\
&=N\beta+ \int_{\R^3}\big(V(x)-\beta\big)\rho_\gamma dx+(a_N^*-a)\int_{\R^3}\rho_\gamma^{\f53}dx\\
&\geq N\beta- \int_{\R^3}\big(\beta-V(x)\big)_+\rho_\gamma dx+(a_N^*-a)\int_{\R^3}\rho_\gamma^{\f53}dx\\
&\geq N\beta-\f25\left(\f 35\right)^{\f 32}\f{1}{(a_N^*-a)^\f32}\int_{\R^3}\big(\beta-V(x)\big)_+^{\f 52}dx,
\end{split}
\end{equation}
where $(\cdot)_+=\max\{\cdot,0\}$ denotes the positive part. Since $V(x)=\omega_1\big(\sqrt{x_1^2+x_2^2}-A\big)^2+\omega_2{x_3^2}$, where $\omega_1>0$ and $\omega_2>0$,  we derive  that for sufficiently small $\beta>0$,
\begin{equation}\label{3-1}
\begin{split}
&\quad\int_{\R^3}\big(\beta-V(x)\big)_+^{\f 52}dx\\
&=	\int_{\R^3}\Big[\beta-\Big(\omega_1\Big(\sqrt{x_1^2+x_2^2}-A\Big)^2+\omega_2x_3^2\Big)\Big]_+^{\f 52}dx\\
&=2\pi\int_{-\sqrt{\f\beta {\omega_2}}}^{\sqrt{\f\beta {\omega_2}}}dx_3\int_{A-\sqrt{\f1 {\omega_1}(\beta-{\omega_2}x_3^2)}}^{A+\sqrt{\f1 {\omega_1}(\beta-\omega_2x_3^2)}}\Big[\beta-\l(\omega_1\left(r-A\right)^2+\omega_2x_3^2\r)\Big]^{\f 52}rdr\\
&=2\pi\int_{-\sqrt{\f\beta {\omega_2}}}^{\sqrt{\f\beta {\omega_2}}}dx_3\int_{-\f\pi2}^{\f\pi2}\l(\beta-{\omega_2}x_3^2\r)^{\f 52}\Big(A+\sqrt{\f{\beta-{\omega_2}x_3^2} {\omega_1}}\sin\theta\Big)\sqrt{\f{\beta-{\omega_2}x_3^2}{\omega_1}}\cos^6\theta d\theta\\
&\leq C\int_{-\sqrt{\f\beta {\omega_2}}}^{\sqrt{\f\beta {\omega_2}}}(\beta-\omega_2x_3^2)^3dx_3\leq C_1\beta^{\f 72},
\end{split}
\end{equation}
where we change the variable ${r}=A+\sqrt{\f1 {\omega_1}(\beta-\omega_2x_3^2)}\sin\theta$ with $-\f \pi2\leq\theta\leq\f \pi2$ in the third equality.
The lower bound of \eqref{3.2} then follows from \eqref{lowerea} and \eqref{3-1} by taking $\beta=\f 12\Big(\f{N}{\f25\left(\f35\right)^{\f 32}C}\Big)^{\f 25}(a_N^*-a)^\f 35>0$.

We next prove the upper bound of \eqref{3.2}. For any $\tau>0$ and $y_0\in\R^3$, set
\begin{equation*}
u_i^\tau(x)=\tau^{\f 32}Q_i\big(\tau(x-y_0)\big),\quad  i=1,\cdots,N ,
\end{equation*}
where $\gamma^{(N)}= \sum_{i=1}^N|Q_i\rangle\langle Q_i|$ is an optimizer of $a_N^*$ satisfying $\langle Q_i,Q_j\rangle_{L^2(\R^3)}=\delta_{ij}$ for $i,\,j=1,\cdots,N$. We then obtain that $\left(u_1^\tau,\cdots,u_N^\tau\right)$ is orthonormal in $L^2(\R^3)$, which implies from \eqref{Q} that $\gamma_\tau:=\sum_{i=1}^N|u_i^\tau\rangle\langle u_i^\tau|$ is in the admissible set of $I_a(N)$. Direct calculations  yield that
\begin{equation}\label{3.3}
\begin{split}
\operatorname{Tr}(-\Delta\gamma_\tau)-a\int_{\R^3}\rho_{\g_\tau}^{\f53}dx&=\tau^2\left(\operatorname{Tr}(-\Delta\gamma^{(N)})-a\int_{\R^3}\rho_{\g^{(N)}}^{\f 53}dx\right)\\
&=\tau^2\left(a_N^*-a\right)\int_{\R^3}\rho_{\g^{(N)}}^{\f 53}dx,
\end{split}
\end{equation}
where $\rho_{\g_\tau}:=\sum_{i=1}^N|u_i^\tau|^2$ and $\rho_{\g^{(N)}}:=\sum_{i=1}^N|Q_i|^2$.
By the exponential decay \eqref{Q} of $Q_i$, we derive that
\begin{equation}\label{3.4}
\begin{split}
\operatorname{Tr}\big(V(x)\g_\tau\big)=	 &\int_{\R^3}\l(\omega_1\left(r-A\right)^2+\omega_2x_3^2\r)\rho_{\g_\tau} dx\\
=&\int_{\R^3}\Big[\omega_1\big(\sqrt{\left({x_1}/\tau+A\right)^2+\left({x_2}/\tau\right)^2}-A\big)^2+\omega_2\l({x_3}/\tau\r)^2\Big]\rho_{\g^{(N)}} dx\\
\leq& \tau^{-2}\int_{\R^3}\Big[\omega_1(x_1^2+x_2^2)+\omega_2x_3^2\Big]\rho_{\g^{(N)}}dx\leq\f{C}{\tau^2},
\end{split}
\end{equation}
where we have chosen $y_0=(A,0,0)\in\R^3$ for the second equality. It thus follows from \eqref{3.3} and \eqref{3.4} that
\begin{equation*}
I_a(N)\leq C'(a_N^*-a)\tau^2+\f{C}{\tau^2},
\end{equation*}
which gives the upper bound of \eqref{3.2} by taking $\tau=(a_N^*-a)^{-\f14}>0$. The proof of Lemma \ref{ree} is therefore complete.\qed

Applying Lemma \ref{ree}, we have the following estimates of $\int_{\R^3}\rho_{\g_a}^{\frac{5}{3}}dx$ as $a\nearrow a^*_N$, where $\g_a$ is a minimizer of $I_a(N)$.

\begin{Lemma}\label{erho}
Suppose $V(x)\ge 0$ is given by \eqref{eV}, and let $N\in\mathbb{N}^+$ be chosen such that
$a_N^*$ admits a full-rank optimizer.
 Assume $\g_{a}=\sum_{i=1}^N|u^a_i\rangle\langle u^a_i|$ is a minimizer of $I_a(N)$, where $u_i^a\in\mathcal{H}$ satisfies $\langle u_i^a,u_j^a\rangle_{L^2(\R^3)}=\delta_{ij}$ for $i,\,j=1,\cdots,N$. Then there exists a positive constant $K$, independent of $a>0$, such that
\begin{equation}\label{re}
0<K(a_N^*-a)^{-\f13}\leq\int_{\R^3}\rho_{\g_a}^{\f 53}dx\leq \f 1K(a_N^*-a)^{-\f12 }\ \ \text{as}\ \ a\nearrow a_N^*.
\end{equation}
\end{Lemma}

\noindent{\bf Proof.}
Since $V(x)\geq0$, it follows from \eqref{a*} that
\begin{equation*}
I_a(N)=E_a(\gamma_a)\geq(a_N^*-a)\int_{\R^3}\rho^{\f53}_{\g_a}dx,
\end{equation*}
which gives the upper bound of \eqref{re} by using  Lemma \ref{ree}.
		
As for the lower bound of \eqref{re}, we choose $0<b<a<a_N^*$ so that
\begin{equation*}
I_b(N)\leq E_b(\gamma_a)=I_a(N)+(a-b)\int_{\R^3}\rho_{\gamma_a}^{\f 53}dx.
\end{equation*}
It then follows from Lemma \ref{ree} that
\begin{equation*}
\begin{aligned}
\int_{\R^3}\rho_{\gamma_a}^{\f 53}dx&\geq \f{	I_b(N)-	I_a(N)}{a-b}\geq \f{C_3(a_N^*-b)^{\f35}-C_4(a_N^*-a)^{\f  12}}{a-b}\\
&=(a_N^*-a)^{-\f13}\f{C_3C_0^{\f35}\Big[1+\f{(a_N^*-a)^{\f16}}{C_0}\Big]^{\f35}-C_4}{C_0}\ \ \hbox{as}\ \ a\nearrow a_N^*,
\end{aligned}
\end{equation*}
by taking $b=a-C_0(a_N^*-a)^{\f 56}\in(0,a)\subset(0,a_N^*)$. One can choose  $C_0>0$  large enough that $C_3C_0^{\f 35}>2C_4$, and hence
\begin{equation*}
\int_{\R^3}\rho_{\g_a}^{\f 53}dx\geq C(a_N^*-a)^{-\f13}\ \ \hbox{as}\ \ a\nearrow a_N^*,
\end{equation*}
which gives the lower bound of \eqref{re}. This completes the proof of Lemma \ref{erho}.\qed

Employing the blow-up analysis, we next address the following convergence of minimizers for $I_a(N)$ as $a\nearrow a_N^*$.

\begin{Lemma}\label{3.3L}
Suppose $V(x)\ge 0$ is given by \eqref{eV}, and let $N\in\mathbb{N}^+$ be chosen such that  any optimizer of $a_N^*$ is full-rank.
 Assume $\g_{a}=\sum_{i=1}^N|u^a_i\rangle\langle u^a_i|$ is a minimizer of $I_a(N)$, where $u_i^a\in\mathcal{H}$ satisfies \eqref{NLS} and  $\big\langle u_i^a,u_j^a\big\rangle_{L^2(\R^3)}=\delta_{ij}$ for $i,\,j=1,\cdots,N$. Define $\epsilon_a>0$ by
\begin{equation}\label{et}
\epsilon_a^{-2}:=\operatorname{Tr} (-\Delta \g_{a})>0,\ \ 0<a<a_N^*.
\end{equation}
Then we have
\begin{enumerate}
\item[(1)] The parameter $\epsilon_a>0$ satisfies
\begin{equation}\label{epsilon0}
\epsilon_a\to0\ \ \text{as}\ \  a\nearrow a^*_N.
\end{equation}

\item[(2)] There exist a sequence $\{y_{\epsilon_a}\}\subset\R^3$ and constants  $R_0,\,\eta>0$ such that the sequence
\begin{equation}\label{wia}
\bar w^a_i=\epsilon_a^{\f 32}u_i^a(\epsilon_a x+\epsilon_ay_{\epsilon_a}),\ \,i=1,\cdots,N,\ \ \bar{\g}_a:=\sum_{i=1}^N|\bar w^a_i\rangle\langle \bar w^a_i|
\end{equation}
satisfies
\begin{equation}\label{liminf}
\liminf_{a\nearrow a_N^*}\int_{B_{R_0}(0)}\rho_{\bar{\g}_a}dx\geq\eta>0,
\end{equation}
where $\rho_{\bar{\g}_a}:=\dsum_{i=1}^N|\bar w^a_i|^2$, and
\begin{equation}\label{tra}
\operatorname{Tr}(-\Delta\bar\gamma_{ a})\equiv 1,\quad\int_{\R^3}\rho_{\bar\gamma_a}^{\f 53}dx\to\f {1}{a_N^*}\ \ \text{as}\ \ a\nearrow a_N^*.
\end{equation}
						
\item[(3)] For any sequence $\{a_n\}$ satisfying $a_n\nearrow a_N^*$ as $n\to\infty$, there exists a subsequence, still denoted by $\{a_n\}$, of $\{a_n\}$ such that
\begin{equation}\label{zb}
\bar x_n:=\epsilon_{a_n}y_{\epsilon_{a_n}}\to \bar x_0 \ \ \text{as}\ \ n\to\infty,
\end{equation}
where $\bar x_0\in\R^3$ is a global minimum point of $V(x)$, $i.e.$,  $V(\bar x_0)=0$. Moreover, we have for $i=1,\cdots,N$,
\begin{equation}\label{sc}
\bar w_i^{a_n}:=\epsilon_{a_n}^{\f 32} u_i^{a_n}(\epsilon_{a_n} x+\epsilon_{a_n}y_{\epsilon_{a_n}})\to\bar w_i\  \text{ strongly in}\ H^1(\R^3)\  \text{ as }\  n\to\infty,
\end{equation}
where $\bar \g:=\sum_{i=1}^N| \bar w_i\rangle\langle \bar w_i|$ is an optimizer of $a_N^*$.
\end{enumerate}
\end{Lemma}

\noindent{\bf Proof.} (1). On the contrary, suppose that \eqref{epsilon0} is false. Applying Lemma \ref{ee}, it then follows from \eqref{a*} that there exists a sequence $\{a_n\}$, where $a_n\nearrow a_N^*$ as $n\to\infty$, such that  $\big\{u_i^{a_n}\big\}$ is bounded uniformly in $\mathcal{H}$ as $n\to\infty$. By Lemma \ref{ce}, we thus derive that there exist a subsequence, still denoted by $\{\g_{a_n}\}$, of $\{\g_{a_n}\}$ and $u_i^{0}\in\H$ such that for $i=1,\cdots,N$,
\begin{equation*}
u_i^{a_n}\rightharpoonup u_i^{0}\ \text{ weakly in }\ \H,\ \   u_i^{a_n}\to u_i^{0}\ \text{ strongly in }\ L^q(\R^3)\ \text{ as }\ n\to\infty,\ \,2\leq q<6,
\end{equation*}
which implies that
\begin{equation*}
\langle u_i^{0},u_i^{0}\rangle_{L^2(\R^3)}=\delta_{ij},\ \ i,j=1,\cdots,N,
\end{equation*}
and
\begin{equation*}
\rho_{\g_{a_n}}:=\sum_{i=1}^N|u_i^{a_n}|^2\to \rho_{\g_0}:=\sum_{i=1}^N|u_i^{0}|^2\ \ \hbox{strongly in}\ \ L^r(\R^3)\ \ \hbox{as}\ \ n\to\infty,\ \ 1\leq r<3,
\end{equation*}
together with $\g_0:=\sum_{i=1}^N|u_i^{0}\rangle\langle u_i^{0}|$.
By the weakly lower semi-continuity, one can then derive from Theorem \ref{groundstate} (2) and  Lemma \ref{ee} that
\begin{equation*}
0=I_{a_N^*}(N)\leq E_{a_N^*}(\g_0)\leq\lim_{n\to\infty}E_{a_n}(\g_{a_n})=\lim_{n\to\infty}I_{a_n}(N)=0.
\end{equation*}
This implies that $\g_0$ is a minimizer of $I_{a_N^*}(N)$, which is impossible in view of the fact that  $I_{a_N^*}(N)$ cannot be achieved. Therefore, $\epsilon_a\to0$ as $a\nearrow a_N^*$, and hence \eqref{epsilon0} holds true.
		
(2). Denote
\begin{equation}\label{3-10}
\check w_i^a:=\epsilon_a^{\f 32}u_i^a(\epsilon_a x),\ \ i=1,\cdots,N,\ \ \check \g_a:=\sum_{i=1}^N|\check w_i^a\rangle\langle\check  w_i^a|,
\end{equation}
where $\epsilon_a>0$ is defined by \eqref{et}. Applying Lemma \ref{ree},  we obtain from  \eqref{a*} and \eqref{et}
that
\begin{equation*}
0\leq\operatorname{Tr}(-\Delta \g_a)-a\int_{\R^3}\rho^{\f 53}_{\g_a}dx=\epsilon_a^{-2}-a\int_{\R^3}\rho^{\f 53}_{\g_a}dx\leq I_a(N)\to0\ \ \text{as}\ \ a\nearrow a_N^*,
\end{equation*}
which implies that
\begin{equation}\label{3-15}
a\epsilon_a^2\int_{\R^3}\rho^{\f 53}_{\g_a}dx\to 1\ \text{ as }\ a\nearrow a_N^*.
\end{equation}
One can then derive from \eqref{3-10} that
\begin{equation}\label{rn0}
\operatorname{Tr}(-\Delta \check \g_a)=1,\ \ \hbox{and}\ \
\int_{\R^3}\rho_{\check\gamma_a}^{\f 53}dx\to\frac{1}{a_N^*}\ \ \hbox{as}\ \ a\nearrow a_N^*,
\end{equation}
where $\rho_{\check\g_a}:=\sum_{i=1}^N|\check w_i^a|^2$. Moreover, we have
\begin{equation}\label{r53}
0<\f 1{2a_N^*}\leq\int_{\R^3}\rho_{\check\gamma_a}^{\f 53}dx\leq \f2{a_N^*}\ \text{ as }\ a\nearrow a_N^*.
\end{equation}

We now claim that  there exist a sequence $\{y_{\epsilon_a}\}\subset\R^3$ and constants  $R_0,\,\eta>0$ such that
\begin{equation}\label{liminf2}
\liminf_{a\nearrow a_N^*}\int_{B_{R_0}(y_{\epsilon_a})}\rho_{\check{\g}_a}dx\geq\eta>0.
\end{equation}
In fact, if \eqref{liminf2} is false, then for any $R>0$, there exists a sequence $\{a_n\}$ satisfying $a_n\nearrow a_N^*$ as $n\to\infty$ such that
\begin{equation*}
\lim_{n\to\infty}\sup_{y\in\R^3}\int_{B_{R}(y)}\rho_{\check{\g}_{a_n}}dx=0,
\end{equation*}
which implies that
\begin{equation*}
\lim_{n\to\infty}\sup_{y\in\R^3}\int_{B_{R}(y)}|\check w_i^{a_n}|^2dx=0\ \text{ for }\ i=1,\cdots,N.
\end{equation*}
We then obtain from \cite[Lemma 1.21]{Willem} that
\begin{equation*}
\check w^{a_n}_i\to0\ \ \text{strongly in }\ L^q(\R^3)\ \text{ as }\ n\to\infty\ \text{ for }\ 2\leq q< 6,\ \,i=1,\cdots,N,
\end{equation*}
which gives that
\begin{equation*}
\int_{\R^3}\rho_{\check\gamma_{a_n}}^{\f 53}dx\to 0\ \text{ as }\ n\to\infty.
\end{equation*}
This is however a contradiction in view of  \eqref{r53}.  We thus obtain that \eqref{liminf2} holds true, which further  gives \eqref{liminf}. Finally, we also note from \eqref{wia} and \eqref{3-15} that \eqref{tra} holds true.

(3). On the contrary,  suppose that \eqref{zb} is false. Then there exist a constant $\delta>0$ and a sequence  $\{a_n\}$ satisfying $a_n\nearrow a_N^*$ as $n\to\infty$ such that
\begin{equation*}
V(\epsilon_{a_n}y_{\epsilon_{a_n}})\geq \delta>0\ \ \hbox{as}\ \ n\to\infty.
\end{equation*}
We then derive from \eqref{liminf} that
\begin{equation}\label{vc}
\begin{split}
&\lim_{n\to\infty}\int_{\R^3} V(\epsilon_{a_n}x+\epsilon_{a_n}y_{\epsilon_{a_n}})\rho_{\bar\g_{ a_n}}dx\\
\geq&\lim_{n\to\infty}\int_{B_{R_0}(0)}V(\epsilon_{a_n}x+	\epsilon_{a_n}y_{\epsilon_{a_n}})\rho_{\bar{\g}_{a_n}}dx\\
\geq& \f{\delta}{2}\lim_{n\to\infty}\int_{B_{R_0}(0)}\rho_{\bar{\g}_{a_n}}dx\geq\f{\delta}{2}\eta.
\end{split}
\end{equation}
However,  it follows from \eqref{a*}, \eqref{eV} and Lemma \ref{ree} that
\begin{equation*}
0\leq \int_{\R^3}V(\epsilon_{a_n}x+\epsilon_{a_n}y_{\epsilon_{a_n}})\rho_{\bar{\g}_{a_n}}dx=\int_{\R^3}V(x)\rho_{\g_{ a_n}}dx\leq I_{a_n}(N)\to 0\ \text{ as }\ n\to\infty,
\end{equation*}
which contradicts with \eqref{vc}. Hence, \eqref{zb} holds true.

In the following, we focus on the proof of \eqref{sc}. Let $\{a_n\}$ be the convergent subsequence given by \eqref{zb}. We then obtain   from \eqref{et} and \eqref{wia} that $\{\bar w_i^{a_n}\}$ is bounded uniformly in $H^1(\R^3)$, and  hence up to a subsequence if necessary,  there exists a function $\bar w_i\in H^1(\R^3)$ such that for $i=1,\cdots,N$,
\begin{equation}\label{weak}
\bar w_i^{a_n}\rightharpoonup \bar  w_i \ \text{ weakly in }\ H^1(\R^3)\ \text{ as }\ n\to\infty,
\end{equation}
and
\begin{equation*}
\bar w_i^{a_n}\to\bar  w_i\  \text{ strongly in }\ L_{loc}^p(\R^3)\  \text{ as }\ n\to\infty,\ \,2\leq p<6.
\end{equation*}
Denoting $\bar\g:=\sum_{i=1}^N|\bar w_i\rangle\langle \bar w_i|$, it then follows from \eqref{liminf} that
\begin{equation*}
\rho_{\bar\g}:=\sum_{i=1}^N|\bar w_i|^2\not\equiv0.
\end{equation*}
The rest of the proof can be divided into the following three steps.

\vspace {.1cm}	

 \textit{Step 1.} We first claim that
\begin{equation}\label{An}
\{A_n\}:=\l\{\f53 a_N^*\rho_{\bar\g_{{a_n}}}^{\f23}\r\} \text{ is a maximizing sequence of the best constant } L_N^*,
\end{equation}
where $\rho_{\bar\g_{a_n}}:=\sum_{i=1}^N|\bar w_i^{a_n}|^2$, and
\begin{equation}\label{3-4}
L_N^*:=\sup_{0\leq A\in L^{\f52}(\R^3)\backslash\{0\}}\f{\sum_{j=1}^N|\lambda_j(-\Delta-A)|}{\int_{\R^3} A^{\f52}(x)dx}
\end{equation}
is attainable (cf. \cite[Corollary 2]{FGL24}). Here $\lambda_j(-\Delta-A)$ denotes the $j$-th min-max level of $-\Delta-A$ in $L^2(\R^3)$, which equals to the $j$-th negative eigenvalue of $-\Delta-A$ in $L^2(\R^3)$ if it exists, and zero otherwise.

By the  min-max principle \cite[Theorem 12.1]{Lieb01}, one can deduce from \eqref{tra} that
\begin{equation}\label{la}
\begin{split}
\sum_{j=1}^N\lambda_j(-\Delta-A_n)&\leq \sum_{j=1}^N\Big\langle(-\Delta-A_n)\bar w_j^{a_n},\bar w_j^{a_n}\Big\rangle\\
&=\dsum_{j=1}^N\int_{\R^3}\Big(\big|\nabla \bar w_j^{a_n}\big|^2-\f53 a_N^*\rho_{\bar{\g}_{a_n}}^{\f 23}\big|\bar w_j^{a_n}\big|^2\Big)dx\\
&=\operatorname{Tr}(-\Delta \bar{\g}_{a_n})-\f 53 {a_N^*}\int_{\R^3}\rho_{\bar{\g}_{a_n}}^{\f53}dx\\
&=-\f23+o(1)\ \text{ as }\ n\to\infty,
\end{split}
\end{equation}
where $A_n$ is defined by \eqref{An}. Recall from \cite[Lemma 6]{FGL21} that
\begin{equation}\label{al}
a_N^*(L_N^*)^{\f23}=\f35\l(\f25\r)^{\f23}.
\end{equation}
We then obtain from \eqref{tra}, \eqref{3-4} and \eqref{la} that
\begin{equation}\label{c2}
\begin{split}
L_N^*\geq\f{\sum_{j=1}^N|\lambda_j(-\Delta-A_n)|}{\int_{\R^3} A_n^{\f52}(x)dx}
&=\l(\f53 a_N^*\r)^{-\f52}\f{\sum_{j=1}^N|\lambda_j(-\Delta-A_n)|}{\int_{\R^3}\rho_{\bar\g_{a_n}}^{\f53} dx}\\
&\geq \l(\f53 a_N^*\r)^{-\f52}\f{\f23}{\f{1}{a_N^*}}+o(1)\\
&=L_N^*+o(1)\ \text{ as }\ n\to\infty,
\end{split}
\end{equation}
which implies that  $\{A_n\}$  is a maximizing sequence of the best constant  $L_N^*$, and hence the claim \eqref{An} holds true.

\vspace {.1cm}	
\textit{Step 2.} The purpose of this step is to prove that
\begin{equation}\label{l53}
\rho_{{\bar{\g}_{a_n}}}\to\rho_{\bar\g}\  \text{ strongly in }\  L^{\f53}(\R^3)\  \text{ as }\  n\to\infty.
\end{equation}
Applying Br\'ezis-Lieb Lemma \cite{Willem}, it suffices to prove that $\int_{\R^3}\rho_{\bar \g}^{\f53}dx=\f 1{a_N^*}$ in view of \eqref{tra}.
		
On the contrary,  suppose that
\begin{equation*}
\lim_{n\to\infty}\int_{\R^3}\rho_{\bar{\g}_{a_n}}^{\f53}dx=\f{1}{a_N^*}>\al:=\int_{\R^3}\rho_{\bar \g}^{\f53}dx.
\end{equation*}
Since it follows from \eqref{liminf} that $\al>0$, we  obtain from the classical dichotomy    that there exist a subsequence, still denoted by $\{\rho_{\bar{\g}_{a_n}}\}$, of $\{\rho_{\bar{\g}_{a_n}}\}$ and a sequence $\{R_n\}$ satisfying $R_n\to\infty$ as $n\to\infty$ such that
\begin{equation}\label{3-5}
0<\lim_{n\to\infty}\int_{|x|<R_n}\rho_{\bar\g_{{a_n}}}^{\f53} dx=\al\ \text{ and }\ \lim_{n\to\infty}\int_{R_n<|x|<2R_n}\rho_{\bar\g_{{a_n}}}^{\f53} dx=0.
\end{equation}
One can then derive from  \cite[Lemma 17]{FGL24} that there exists some $s\in\{0,\,1,\,\dots,N\}$ such that
\begin{equation}\label{2fr}
\begin{split}
\sum_{j=1}^N|\lambda_j(-\Delta-A_n)|=&\sum_{j=1}^s|\lambda_j(-\Delta-A_n\mathds{1}_{B_{R_n}})|\\
&+\sum_{j=1}^{N-s}|\lambda_j(-\Delta-A_n\mathds{1}_{\R^3\backslash B_{2R_n}})|+o(1) \ \ \mbox{as} \ \ n\to\infty,
\end{split}
\end{equation}
where $A_n$ is as in \eqref{An}.		
		
Recall from \eqref{frlt} that the best constant $L_s^*$ of the finite-rank Lieb-Thirring inequality is defined as
\begin{equation*}
L_s^*:=\sup_{0\leq A\in L^{\f52}(\R^3)\backslash\{0\}}\f{{\sum_{j=1}^s\big|\lambda_j\big(-\Delta-A(x)\big)\big|}}{{\int_{\R^3} A^{\f52}(x)dx}},\ \,\forall
 s\in\N^+\cup \{0\},
\end{equation*}
where $s=0$ means that $L_s^*=0$.
It is clear that $L_s^*$ is increasing with respect to the nonnegative integer $s$. We thus conclude from \eqref{tra}, \eqref{An}, \eqref{c2}, \eqref{3-5} and \eqref{2fr} that for some $s\in\{0,1,\dots,N\}$,
\begin{equation}\label{alc}
\begin{split}
&\ \ \l(\f53 a_N^*\r)^{\f52}\f{1}{a_N^*} L_N^*\\
&=L_N^*\lim_{n\to\infty}\int_{\R^3}A_n^{\f52}dx\\
&=\lim_{n\to\infty} \sum_{j=1}^N\l|\lambda_j(-\Delta-A_n)\r|\\
&=\lim_{n\to\infty}\sum_{j=1}^s\big|\lambda_j(-\Delta-A_n\mathds{1}_{B_{R_n}})\big|+\lim_{n\to\infty}\sum_{j=1}^{N-s}\big|\lambda_j(-\Delta-A_n\mathds{1}_{\R^3\backslash B_{2R_n}})\big|\\
&\leq L_s^*\lim_{n\to\infty}\int_{\R^3}(A_n\mathds{1}_{B_{R_n}})^{\f52}dx+L_{N-s}^*\lim_{n\to\infty}\int_{\R^3}(A_n\mathds{1}_{\R^3\backslash B_{2R_n}})^{\f52}dx\\
&=\Big(\f53 a_N^*\Big)^{\f52}\Big[L_s^*\al+L_{N-s}^*\Big(\f{1}{a_N^*}-\al\Big)\Big]\leq \Big(\f53 a_N^*\Big)^{\f52}\f{1}{a_N^*} L_N^*,
\end{split}
\end{equation}
where we have used the fact that $L_s^*$ is increasing with respect to the nonnegative integer  $s\in\{0,1,\dots,N\}$. Since $\f{1}{a_N^*}>\al$,  one can deduce from \eqref{alc} that $L_s^*=L_{N-s}^*=L_N^*$.
		
On the other hand, since $(N-s)+s=N$, we have either $2(N-s)\leq N$ or $2s\leq N$. Recalling from \cite[Theorem 4]{FGL24} that $L_{2N}^*>L_{N}^*$ for all $N\in\N^+$, it then follows from the monotonicity of $L_s^*$ that either
$$L^*_{N-s}<L^*_{2(N-s)}\leq L^*_N,$$
or
$$L^*_s<L^*_{2s}\leq L^*_N,$$
a contradiction. We thus conclude that $	 \int_{\R^3}\rho_{\bar\g}^{\f53}dx=\lim_{n\to\infty}\int_{\R^3}\rho_{\bar\g_{{a_n}}}^{\f53}dx=\f{1}{a_N^*}$, which implies that \eqref{l53} holds true.
		
\vspace {.1cm}	

\textit{Step 3.} Following Fatou's Lemma, we obtain from \eqref{a*},  \eqref{tra}, \eqref{weak} and \eqref{l53} that
\begin{equation*}
\begin{split}
1&=a_N^*\lim_{n\to\infty}\int_{\R^3}\rho_{\bar\g_{{a_n}}}^{\f53}dx=\lim_{n\to\infty}\operatorname{Tr}(-\Delta \bar{\g}_{a_n})\\
&\geq \operatorname{Tr}(-\Delta \bar \g)\geq a_N^*\|\bar \g\|^{-\f23}\int_{\R^3}\rho_{\bar \g}^{\f53}dx\geq1,
\end{split}
\end{equation*}
where we have used the fact that $\|\bar \g\|\leq\ds\liminf_{n\to\infty}\|\bar{\g}_{a_n}\|=1 $. We thus conclude that $\|\bar \g\|=1$, and $\bar \g$ is an optimizer of $a_N^*$.
Moreover, we also have
\begin{equation}\label{nstr}
\n \bar w_i^{a_n}\to\n \bar w_i\ \text{ strongly in }\ L^2{(\R^3)}\  \text{ as }\ n\to\infty,\ \,i=1,\cdots,N.
\end{equation}	
Since Rank$(\bar \g)=N$, we  obtain from \eqref{gN} that
\begin{equation*}
\int_{\R^3}\rho_{\bar \g} dx=\|\bar \g\|\text{Rank}(\bar \g)=N.
\end{equation*}
One can then  deduce from \eqref{weak} and the Br\'ezis-Lieb lemma that
\begin{equation*}
\rho_{\bar{\g}_{a_n}}\to\rho_{\bar\g}\ \text{ strongly in}\ L^1(\R^3)\ \text{ as }\ n\to\infty,
\end{equation*}
which further implies that
\begin{equation}\label{w2}
\bar w_i^{a_n}\to \bar w_i\ \text{ strongly in}\ L^2(\R^3)\ \text{ as }\ n\to\infty,\ \,i=1,\cdots,N.
\end{equation}
We conclude from  \eqref{nstr} and \eqref{w2} that \eqref{sc} holds true. This therefore completes the proof of Lemma \ref{3.3L}. \qed

\begin{Lemma}\label{Vll}
Under the assumptions of Lemma \ref{3.3L}, let $\{a_n\}$ be the subsequence obtained in Lemma \ref{3.3L} (3). Then for any $R>0$, there exists a constant $C_0(R)>0$, independent of $a_n$, such that
\begin{equation}\label{l3.5}
\lim_{n\to\infty}\f{1}{\epsilon_{a_n}^2}\int_{B_R(0)}V(\epsilon_{a_n}x+\epsilon_{a_n}y_{\epsilon_{a_n}} )\rho_{\bar\g_{a_n}}dx\geq C_0(R),
\end{equation}
where $\epsilon_{a_n}>0$ and $\bar\g_{a_n}$ are defined by \eqref{et} and \eqref{wia}, respectively.
\end{Lemma}

\noindent{\bf Proof.}	By the definition of $V(x)$ in \eqref{eV}, we get that for $0<c\leq \min\{\omega_1,\omega_2\}$,
\begin{equation*}
V(x)-c(|x|-A)^2
=(\omega_1-c)(r-A)^2+(\omega_2-c)x_3^2+2Ac(|x|-r)\geq0,
\end{equation*}
where $x=(x_1,x_2,x_3)\in \R^3$,  and $r=\sqrt{x_1^2+x_2^2}$. Thus, it suffices to prove that for any $R>0$, there exists a constant $C_0(R)>0$, independent of $a_n>0$, such that
\begin{equation}\label{Vl}
\lim_{n\to\infty}\f{1}{\epsilon_{n}^2}\int_{B_R(0)}(|\epsilon_{n}x+\epsilon_{n}y_{\epsilon_{n}} |-A)^2\rho_{\bar\g_{a_n}}dx\geq C_0(R),
\end{equation}
where $\epsilon_n:=\epsilon_{a_n}>0$, and $A>0$ is as in \eqref{eV}.
		
Note that
\begin{equation*}
\begin{split}
(|\epsilon_{n}x+\epsilon_{n}y_{\epsilon_{n}}|-A)^2&=\epsilon_{n}^2\left(|x+y_{\epsilon_{n}}|-\f{A}{\epsilon_{n}}\right)^2\\
&=\epsilon_{n}^2\Big(\sqrt{|x|^2+|y_{\epsilon_{n}}|^2}\sqrt{1+\f{2x\cdot y_{\epsilon_{n}}}{|x|^2+|y_{\epsilon_{n}}|^2}}-\f{A}{\epsilon_{n}}\Big)^2.
\end{split}
\end{equation*}
Since $\lim_{n\to\infty}\epsilon_{n}y_{\epsilon_{n}}=\bar x_0$ holds for $|\bar x_0|=A>0$, we obtain that $\lim_{n\to\infty}|y_{\epsilon_{n}}|=\infty$. This implies that for any given $R>0$,
$ \ds\lim_{n\to\infty}\f{2x\cdot y_{\epsilon_{n}}}{|x|^2+|y_{\epsilon_{n}}|^2}=0$ holds uniformly for $x\in B_R(0)$. By Taylor's expansion, we then derive that for all $x\in B_R(0)$,
\begin{equation}\label{epV}
\begin{split}
&\f{1}{\epsilon_{n}^2}\big(|\epsilon_{n}x+\epsilon_{n}y_{\epsilon_{n}}|-A\big)^2\\
=&\Big(\sqrt{|x|^2+|y_{\epsilon_{n}}|^2} \sqrt{1+\f{2x\cdot y_{\epsilon_{n}}}{|x|^2+|y_{\epsilon_{n}}|^2}}-\f{A}{\epsilon_{n}}\Big)^2\\
=&\Big[\sqrt{|x|^2+|y_{\epsilon_{n}}|^2}\l(	1+\f{x\cdot y_{\epsilon_{n}}}{|x|^2+|y_{\epsilon_{n}}|^2}+O\big(\f{1}{|y_{\epsilon_{n}}|^2}\Big)\r)-\f{A}{\epsilon_{n}}\Big]^2
\\
=&\Big[\sqrt{|x|^2+|y_{\epsilon_{n}}|^2}+\f{x\cdot y_{\epsilon_{n}}}{\sqrt{|x|^2+|y_{\epsilon_{n}}|^2}}-\f{A}{\epsilon_{n}}+O\Big(\f{1}{|y_{\epsilon_{n}}|}\Big)\Big]^2\text{ \ as \ }n\to\infty.
\end{split}
\end{equation}
For any $x\in\R^3$, the spherical coordinate of $x$ is given by
\begin{equation}
x=(r'\cos\phi,r'\sin\phi \cos\theta,r'\sin\phi \sin\theta)\in\R^3,
\end{equation}
where $r'=|x|\geq0,\ \phi\in[0,\pi]$ and $\theta\in[0,2\pi]$. We denote  $\arg x$ as the angle between $x$ and the positive $x$-axis (namely, $\arg x=\phi$), and $\l<x,y\r>$ as the angle between the vectors $x$ and $y$. Without loss of generality,   assuming that $x_0=(A,0,0)$, it then follows that $\arg y_{\epsilon_{n}}\to0$ as $n\to\infty$. One can thus choose $0<\delta<\f{\pi}{20}$  small enough that
\begin{equation}\label{delta}
0\leq\arg y_{\epsilon_{n}}<\delta\ \text{ as }\ n\to\infty.
\end{equation}
		
Denote for above given $R>0$,
\begin{equation*}
\begin{split}
\Omega_{n}^1 &=\left\{x \in B_{R}(0): \sqrt{|x|^{2}+\left|y_{\epsilon_{n}}\right|^{2}} \leq \frac{A}{\epsilon_{n}}\right\} \\
&=\left\{x \in B_{R}(0):|x|^{2}\leq\Big(\frac{A}{\epsilon_{n}}\Big)^{2}-\left|y_{\epsilon_{n}}\right|^{2}\right\},
\end{split}
\end{equation*}
and
\begin{equation*}
\begin{split}
\Omega_{n}^{2} &=\left\{x \in B_{R}(0): \sqrt{|x|^{2}+\left|y_{\epsilon_{n}}\right|^{2}}>\frac{A}{\epsilon_{n}}\right\}\\
&=\left\{x \in B_{R}(0):\Big(\frac{A}{\epsilon_{n}}\Big)^{2}-\left|y_{\epsilon_{n}}\right|^{2}<|x|^{2}<R^{2}\right\},
\end{split}
\end{equation*}
so that $B_{R}(0)=\Omega_{n}^{1} \cup \Omega_{n}^{2}$ and $\Omega_{n}^{1} \cap \Omega_{n}^{2}=\emptyset$. Since
\begin{equation*}
\left|\Omega_{n}^{1}\right|+\left|\Omega_{n}^{2}\right|=\left|B_{R}(0)\right|=\f43\pi R^{3},
\end{equation*}
there exists a subsequence, still denoted by $\left\{\epsilon_{n}\right\}$, of $\left\{\epsilon_{n}\right\}$ such that
\begin{equation*}
\text { either }\left|\Omega_{n}^{1}\right| \geq \frac{2\pi R^{3}}{3} \text { or }\left|\Omega_{n}^{2}\right| \geq \frac{2\pi R^{3}}{3}.
\end{equation*}
We now consider separately the following two cases for above given $R>0$:

\vspace {.1cm}

\textbf{Case 1.} $\big|\Omega_{n}^{1}\big| \geq \frac{2\pi R^{3}}{3}$. In this case, we have $B_{\frac{R}{\sqrt[3]2}}(0) \subset \Omega_{n}^{1}$. Define
\begin{equation*}
\Omega_{1}:=\left(B_{\frac{R}{\sqrt[3]2}}(0) \backslash B_{\frac{R}{2}}(0)\right) \cap\left\{x: \frac{\pi}{2}+2 \delta<\arg x<\pi\right\} \subset \Omega_{n}^{1},
\end{equation*}
so that
\begin{equation*}
\begin{split}
|\Omega_{1}|&=\int_0^{2\pi}\int_{\frac{R}{2}}^{\frac{R}{\sqrt[3]2}}\int_{\frac{\pi}{2}+2 \delta}^{\pi}(r')^2\sin\phi \,dr'd\phi \,d\theta\\
&=\f{\pi R^3}{4}\big[1-\sin(2\delta)\big]>\f{\pi R^3}{4}\Big(1-\sin\f\pi{10}\Big).
\end{split}
\end{equation*}
It then follows from \eqref{delta} that for any $x\in\Omega_{1}$,
\begin{equation*}
x\cdot y_{\epsilon_n}=|x||y_{\epsilon_n}|\cos\l<x,y_{\epsilon_n}\r><0\ \text{ and }\ |\cos\l<x,y_{\epsilon_n}\r>|>-\cos\left(\f\pi 2+\delta\right)>0,
\end{equation*}
which implies   that  for $x\in\Omega_1$,
\begin{equation}\label{3-6}
\begin{split}
&\sqrt{|x|^2+|y_{\epsilon_n}|^2}+\f{x\cdot y_{\epsilon_n}}{\sqrt{|x|^2+|y_{\epsilon_n}|^2}}-\f{A}{\epsilon_n}+O\Big(\f{1}{|y_{\epsilon_n}|}\Big)\\
\leq&\f{x\cdot y_{\epsilon_n}}{\sqrt{|x|^2+|y_{\epsilon_n}|^2}}+O\Big(\f{1}{|y_{\epsilon_n}|}\Big)\\
\leq&\f{x\cdot y_{\epsilon_n}}{2\sqrt{|x|^2+|y_{\epsilon_n}|^2}}\leq\f{|x||y_{\epsilon_n}|\cos(\f{\pi}{2}+\delta)}{2\sqrt{|x|^2+|y_{\epsilon_n}|^2}}<0\ \ \text{as}\ \ n\to\infty.
\end{split}
\end{equation}
Since $\lim_{n\to\infty}|y_{\epsilon_n}|=\infty$, we obtain from \eqref{epV} and \eqref{3-6} that for $x\in\Omega_1$,
\begin{equation*}
\f{1}{\epsilon_n^2}	\big(|\epsilon_nx+\epsilon_ny_{\epsilon_n}|-A\big)^2\geq \f{|x|^2\cos^2(\f{\pi}{2}+\delta)}{8}\ \ \hbox{as}\ \ n\to\infty.
\end{equation*}
Taking $\delta=\f{\pi}{30}$, it  then follows from \eqref{sc}  that
\begin{equation*}
\begin{split}
&\lim_{n\to\infty}\f{1}{\epsilon_n^2}\int_{B_R(0)}	 (|\epsilon_nx+\epsilon_ny_{\epsilon_n}|-A)^2\rho_{\bar\g_{a_n}}dx\\
\geq &\lim_{n\to\infty}\f{1}{\epsilon_n^2}\int_{\Omega_{1}}	 (|\epsilon_nx+\epsilon_ny_{\epsilon_n}|-A)^2\rho_{\bar\g_{a_n}}dx\\
\geq &\f{\cos^2\f{8\pi}{15}}{8}\int_{\Omega_{1}}|x|^2\rho_{\bar\g} dx:=C_0(R)>0,
\end{split}
\end{equation*}
where $\rho_{\bar\g}=\sum_{i=1}^N|\bar w_i|^2$. This thus proves \eqref{Vl}.

\vspace {.1cm}
\textbf{Case 2.} $\left|\Omega_{n}^{2}\right| \geq \frac{2\pi R^{3}}{3}$. In this case, we have
$$D_{R}:=\Big(B_{R}(0) \backslash B_{\f{R}{\sqrt[3]2}}(0)\Big) \subset \Omega_{n}^{2}.$$
Setting  $\Omega_{2}:=D_{R} \cap\left\{x: 0<\arg x<\frac{ \pi}{2}-2 \delta\right\}\subset \Omega_{n}^{2}$, we  then obtain that
\begin{equation*}
\left|\Omega_{2}\right|=\f{\pi R^3}{3}\big[1-\sin(2\delta)\big]\geq\f{\pi R^3}{3}\l(1-\sin\f\pi{10}\r).
\end{equation*}
The similar argument of  Case 1 yields that for $x\in\Omega_{2}$,
\begin{equation*}
x\cdot y_{\epsilon_n}=|x||y_{\epsilon_n}|\cos\l<x,y_{\epsilon_n}\r>>0\ \text{ and }\ \cos\l<x,y_{\epsilon_n}\r>>\cos\left(\f\pi 2-\delta\right)>0,
\end{equation*}
which implies that for $x\in\Omega_2$,
\begin{equation*}
\begin{split}
&\sqrt{|x|^2+|y_{\epsilon_n}|^2}+\f{x\cdot y_{\epsilon_n}}{\sqrt{|x|^2+|y_{\epsilon_n}|^2}}-\f{A}{\epsilon_n}+O\Big(\f{1}{|y_{\epsilon_n}|}\Big)\\
\geq&\f{x\cdot y_{\epsilon_n}}{\sqrt{|x|^2+|y_{\epsilon_n}|^2}}+O\Big(\f{1}{|y_{\epsilon_n}|}\Big)\\
\geq&\f{x\cdot y_{\epsilon_n}}{2\sqrt{|x|^2+|y_{\epsilon_n}|^2}}\geq\f{|x| |y_{\epsilon_n}|\cos(\f{\pi}{2}-\delta)}{2\sqrt{|x|^2+|y_{\epsilon_n}|^2}}>0\ \ \hbox{as}\ \ n\to\infty.
\end{split}
\end{equation*}
Taking $\delta=\f{\pi}{30}$, we thus obtain from \eqref{sc} and \eqref{epV} that
\begin{equation*}
\begin{split}
&\lim_{n\to\infty}\f{1}{\epsilon_n^2}\int_{B_R(0)}(|\epsilon_nx+\epsilon_ny_{\epsilon_n}|-A)^2\rho_{\bar\g_{a_n}}dx\\
\geq&\lim_{n\to\infty}\f{1}{\epsilon_n^2}\int_{\Omega_{2}}(|\epsilon_nx+\epsilon_ny_{\epsilon_n}|-A)^2\rho_{\bar\g_{a_n}}dx\\
\geq&\f{\cos^2\f{7\pi}{15}}{8}\int_{\Omega_{2}}|x|^2\rho_{\bar\g} dx:=C_0(R)>0,
\end{split}
\end{equation*}
where $\rho_{\bar\g}=\sum_{i=1}^N|\bar w_i|^2$. This also proves \eqref{Vl}, and the proof of Lemma \ref{Vll} is therefore complete.\qed
\vspace{5pt}

Applying Lemma \ref{Vll}, we are now ready to complete the proof of Theorem \ref{ee}.
\vspace{5pt}

\noindent{\bf Proof of Theorem 3.1.}
By Lemma \ref{ree}, it remains to prove that there exists a positive constant $C_1>0$, independent of $a>0$, such that
\begin{equation}\label{rel}
C_1(a_N^*-a)^{\f 12}\leq I_a(N)\ \text{ as }\ a\nearrow a_N^*.
\end{equation}
Indeed,   the proof of Lemma \ref{3.3L} (3) gives that for any sequence $\{\bar \g_{a_n}\}$ satisfying $a_n\nearrow a_N^*$ as $n\to\infty$, there exists a  subsequence, still denoted by $\{\bar \g_{a_n}\}$, of $\{\bar \g_{a_n}\}$ such that
\begin{equation*}
\rho_{{\bar{\g}_{a_n}}}\to\rho_{\bar\g}\ \text{ strongly in}\  L^{\f53}(\R^3) \ \text{ as } \ n\to\infty,
\end{equation*}
where $\bar \g_{a_n}$ is as in \eqref{wia},  and $\bar\g:=\sum_{i=1}^N|\bar w_i\rangle\langle \bar w_i|$ is an optimizer of $a_N^*$ defined by \eqref{a*}. This implies that there exists a constant $M_1>0$, independent of $a_n>0$, such that
\begin{equation*}
\int_{\R^3}\rho_{\bar\g_{a_n}}^{\f53} dx\geq M_1\ \text{ as }\ n\to\infty.
\end{equation*}

On the other hand, we get from \eqref{l3.5} with $R=1$ that there exists a positive  constant $M_2>0$, independent of $a_n>0$,  such that
\begin{equation*}
\int_{B_1(0)}V(\epsilon_{a_n}x+\epsilon_{a_n}y_{\epsilon_{a_n}})\rho_{\bar\g_{a_n}}dx\geq M_2\epsilon_{a_n}^2\ \text{ as }\ n\to\infty.
\end{equation*}
We thus conclude that
\begin{equation*}
\begin{split}
I_{a_n}(N)=E_{a_n}(\gamma_{a_n})=&\f{1}{\epsilon_{a_n}^2}\Big(\operatorname{Tr}(-\Delta \bar\g_{a_n})-a_N^*\int_{\R^3}\rho_{\bar\g_{a_n}}^{\f{5}3}dx\Big)\\
&+\f{a_N^*-a_n}{\epsilon_{a_n}^2}\int_{\R^3}\rho_{\bar\g_{a_n}}^{\f{5}3}dx+	 \int_{\R^3}V(\epsilon_{a_n}x+\epsilon_{a_n}y_{\epsilon_{a_n}})\rho_{\bar\g_{a_n}}dx\\
&\geq  \f{a_N^*-a_n}{\epsilon_{a_n}^2}M_1+M_2\epsilon_{a_n}^2\\
&\geq 2\sqrt{M_1M_2}(a_N^*-a_n)^{\f 12}\ \text{ as }\ n\to\infty,
\end{split}
\end{equation*}
which gives that  \eqref{rel} holds for the subsequence $\{a_n\}$.

Since the above argument is true for any subsequence $\{a_n\}$ satisfying $a_n\nearrow a_N^*$ as $n\to\infty$,  it implies that \eqref{rel} holds for all  $a\nearrow a_N^*$. This completes the proof of Theorem \ref{ee}.\qed

\section{Mass Concentration of Minimizers as $a\nearrow a_N^*$}\label{4}

The purpose of this section is to establish Theorems \ref{1.2T} on the mass concentration behavior of minimizers for $I_a(N)$ as $a\nearrow a_N^*$. Towards this aim, suppose $\g_{a}=\sum_{i=1}^N| u^a_i\rangle\langle u^a_i|$ is a minimizer of  $I_a(N)$, where $u_i^a\in\mathcal{H}$ satisfies \eqref{NLS} and  $\langle u_i^a,u_j^a\rangle_{L^2(\R^3)}=\delta_{ij}$ for $i,\,j=1,\cdots,N$. Applying Theorem \ref{ee}, we then obtain that there exists a constant $K>0$, independent of $a>0$, such that
\begin{equation}\label{4.1}
0<K(a_N^*-a)^{-\f12}\leq\int_{\R^3}\rho_{\g_a}^{\f 53}dx\leq \f 1K(a_N^*-a)^{-\f12 }\ \text{ as }\ a\nearrow a_N^*,
\end{equation}
where $\rho_{\g_a}:=\sum_{i=1}^N|u_i^a|^2$.

Inspired by (\ref{4.1}), we define
\begin{equation}\label{ne}
\varepsilon_a:=(a^*_N-a)^{\f 14}>0,\ \ 0<a<a_N^*.
\end{equation}
It follows  from \eqref{a*} that
\begin{equation}
I_a(N)=E_a(\g_a)\geq\Big(1-\f a{a_N^*}\Big)\operatorname{Tr}(-\Delta\gamma_a)+\int_{\R^3}V(x)\rho_{\g_a}dx.
\end{equation}
Applying Theorem \ref{ee}, this gives that
\begin{equation}\label{4.3}
\operatorname{Tr}(-\Delta\gamma_a)\leq C\varepsilon_a^{-2},\quad\int_{\R^3}V(x)\rho_{\g_a}dx\leq C\varepsilon_a^2\ \ \hbox{as}\ \ a\nearrow a_N^*.
\end{equation}
The  similar argument of  Lemma \ref{3.3L} (2) then yields that there exist a sequence $\{y_{\varepsilon_a}\}\subset \R^3$, and positive constants $R_0$ and $\eta$ such that the sequence
\begin{equation}\label{win}
\hat w^a_i(x):=\varepsilon_a^{\f32}u_i^a(\varepsilon_a x+\varepsilon_ay_{\varepsilon_a}),\ \ i=1,\cdots,N,\ \ \hat\g_{{a}}:=\dsum_{i=1}^N|\hat w^a_i\rangle\langle\hat w^a_i|,
\end{equation}
satisfies
\begin{equation}\label{liminf3}
\liminf_{a\nearrow a_N^*}\int_{B_{R_0}(0)}\rho_{\hat{\g}_a}dx\geq\eta>0,
\end{equation}
where $\rho_{\hat{\g}_a}:=\dsum_{i=1}^N|\hat w^a_i|^2$. We then obtain from \eqref{4.1}, \eqref{ne} and \eqref{4.3} that
\begin{equation}\label{4-1}
\operatorname{Tr}(-\Delta\hat\gamma_a)\leq C,\quad0<K\leq\int_{\R^3}\rho_{\hat\g_a}^{\f 53}dx\leq \f 1K\ \text{ as }\ a\nearrow a_N^*,
\end{equation}
where the positive constants $C$ and $K$ are independent of $a>0$.


\begin{Lemma}\label{4.1L}
Assume $V(x)$ satisfies \eqref{eV}, and let $N\in\mathbb{N}^+$ be chosen such that any optimizer of $a_N^*$ is full-rank.
  Suppose $\g_{{a_n}}=\sum_{i=1}^N| u^{a_n}_i\rangle\langle u^{a_n}_i|$ is a minimizer of $I_{a_n}(N)$ satisfying $a_n\nearrow a_N^*$ as $n\to\infty$, where $u_i^{a_n}\in\mathcal{H}$ satisfies \eqref{NLS} and $\langle u_i^{a_n}, u_j^{a_n}\rangle_{L^2(\R^3)}=\delta_{ij}$ for $i,\,j=1,\cdots,N$. Then we have
\begin{enumerate}
\item[(1)] There exists a subsequence, still denoted by $\{a_n\}$, of $\{a_n\}$ such that
\begin{equation}\label{zA}
\hat x_n:=\varepsilon_{a_n}y_{\varepsilon_{a_n}}\to x_0\ \text{ as }\  n\to\infty,\ \ \varepsilon_{a_n}=(a_N^*-a_n)^{\frac14}>0,
\end{equation}
where  $x_0\in\R^3$ is a global minimum point of $V(x)$, $i.e.$,  $V(x_0)=0$. Moreover, there exists a function $\hat w_i\in H^1(\R^3) $ such that
\begin{equation}\label{4-2}
\hat w_i^{a_n}\to \hat w_i\ \ \hbox{strongly in}\ \  H^1(\R^3)\ \ \hbox{as}\ \ n\to\infty,\ \ i=1,\cdots,N,
\end{equation}
where $\hat w^{a_n}_i$ is defined by \eqref{win}, and $\hat\g:=\sum_{i=1}^N|\hat w_i\rangle\langle \hat w_i|$ is an optimizer of $a_N^*$ defined by \eqref{a*}.

\item[(2)] The $i$-th eigenvalue $\mu_i^{a_n}$ of the operator $H_V^{n}:=-\Delta+V(x)-\frac{5a_n}{3}\rho_{\g_{a_n}}^{\frac{2}{3}}$ in $L^2(\R^3)$ satisfies
\begin{equation}\label{limmu}
\lim_{n\to\infty}\varepsilon_{a_n}^2\mu_i^{a_n}=\mu_i<0,\quad i=1,\cdots,N,
\end{equation}
where $\mu_i$ is the $i$-th eigenvalue of the operator $H_{\hat\g}:=-\Delta-\frac{5a_N^*}{3}\rho_{\hat\g}^{\f23}$ in $L^2(\R^3)$, and $\hat \g=\sum_{i=1}^N|\hat w_i\rangle\langle \hat w_i|$ is as in \eqref{4-2}.

\item[(3)]There exists a constant $C>0$, independent of $a_n>0$, such that for $i=1,\cdots,N$,
\begin{equation}\label{wed}
|\hat w_i^{a_n}|\leq C\e^{-\sqrt{\f12|\mu_i|}|x| }\ \ \hbox{and}\ \ \rho_{\hat\g_{a_n}}\leq Ce^{-\sqrt{2|\mu_N|}|x|}\ \ \text{uniformly in}\ \, \R^3
\end{equation}
as $n\to\infty$, where $\mu_i<0$ is as in \eqref{limmu} for $i=1,\cdots,N$.
\end{enumerate}
\end{Lemma}

\noindent{\bf Proof.} (1). Applying Theorem  \ref{ee} and \eqref{liminf3}, the similar arguments of Lemma \ref{3.3L} (3) yield that both \eqref{zA} and \eqref{4-2} hold true. We omit the details for simplicity.

(2). Since $\g_{a_n}$ is a minimizer of $I_{a_n}(N)$ satisfying  \eqref{NLS}, we have
\begin{equation}\label{4-12}
-\Delta u_i^{a_n}+V(x)u_i^{a_n}-\f53 a_n\rho_{\g_{a_n}}^{\f23}u_i^{a_n}=\mu_i^{a_n}u_i^{a_n}\ \ \text{in}\ \, \R^3,\ \ i=1,\cdots,N,
\end{equation}
where $\mu_i^{a_n}$ is the $i$-th eigenvalue of the operator $H_V^{n}:=-\Delta+V(x)-\frac{5a_n}{3}\rho_{\g_{a_n}}^{\frac{2}{3}}$ in $L^2(\R^3)$. It then follows from \eqref{4-12} that
\begin{equation*}
\mu^{a_n}_i =\int_{\R^3}|\nabla u_i^{a_n}|^2dx +\int_{\R^3}V(x)|u_i^{a_n}|^2dx-\f53 {a_n}\int_{\R^3}\rho_{\g_{a_n}}^{\f 23}|u^{a_n}_i|^2dx,\ \,i=1,\cdots,N,
\end{equation*}
which implies that
\begin{equation*}
\begin{split}		
\sum_{i=1}^N\mu_i^{a_n}&=\operatorname{Tr}\big(-\Delta+V(x)\big)\g_{a_n}-\f53{a_n}\int_{\R^3}\rho_{\g_{a_n}}^{\f 53}dx\\
&=I_{a_n}(N)-\f23 {a_n}\int_{\R^3}\rho^{\f 53}_{\g_{a_n}}dx.
\end{split}
\end{equation*}
Applying Theorem \ref{ee}, we then obtain from  \eqref{win} and \eqref{4-1} that
\begin{equation}\label{limem}
\begin{split}
\varepsilon_{a_n}^2\sum_{i=1}^N\mu_i^{a_n}&=\varepsilon_{a_n}^2\l(I_{a_n}(N)-\f23 {a_n}\int_{\R^3}\rho^{\f 53}_{\g_{a_n}}dx\r)\\
&=-\frac{2a_N^*}{3}\int_{\R^3}\rho_{\hat\g_{a_n}}^{\f53}dx+o(1)\leq -\frac{a_N^*K}{3}\ \text{ as }\ n\to \infty,
\end{split}
\end{equation}
where $K>0$ is as in \eqref{4.1}.

Note from  \eqref{win} and \eqref{4-12} that $\hat w_i^{a_n}$  satisfies  for $i=1,\cdots,N$,
\begin{equation}\label{lmw}
-\Delta \hat w_i^{a_n}+\varepsilon_{a_n}^2V(\varepsilon_{a_n}x +\varepsilon_{a_n}y_{\varepsilon_{a_n}})\hat w_i^{a_n}-\f53 a_n\rho_{\hat{\g}_{a_n}}^{\f 23}\hat w_i^{a_n}=\varepsilon_{a_n}^2\mu^{a_n}_i \hat w_i^{a_n}\ \ \mbox{in}\, \ \R^3.
\end{equation}
We then derive  from \eqref{4-1} and \eqref{lmw} that for $i=1,\cdots,N$,
\begin{equation}\label{infm1}
\begin{split}
&\liminf_{n\to\infty}\varepsilon_{a_n}^2\mu^{a_n}_i\\
=&\liminf_{n\to\infty}\Big[\int_{\R^3}|\nabla \hat w_i^{a_n}|^2dx+\varepsilon_{a_n}^2\int_{\R^3}V(\varepsilon_{a_n}x +\varepsilon_{a_n}y_{\varepsilon_{a_n}})|\hat w_i^{a_n}|^2dx\\
&\qquad\qquad-\f53 a_n\int_{\R^3}\rho_{\hat{\g}_{a_n}}^{\f 23}|\hat w_i^{a_n}|^2dx\Big]\\
\geq&-\limsup_{n\to\infty}\f53a_n\int_{\R^3}\rho_{\hat{\g}_{a_n}}^{\f53}dx\geq-\f {5a_N^*}{3K},
\end{split}
\end{equation}
which implies that up to a subsequence if necessary, $\big\{\varepsilon_{a_n}^2\mu_i^{a_n}\big\}$ is  bounded from below for $i=1,\cdots,N$. It then follows from \eqref{limem}
and \eqref{infm1} that
$$\Big\{\sum_{i=1}^N\varepsilon_{a_n}^2\mu^{a_n}_i\Big\} \text{ is bounded uniformly in }n,$$ which further gives that
\begin{equation*}
\big\{\varepsilon_{a_n}^2\mu^{a_n}_i\big\} \text{ is bounded uniformly in }n,\  \,i=1,\cdots,N.
\end{equation*}
Therefore, there exists a subsequence, still denoted by $\{a_n\}$, of $\{a_n\}$ such that
\begin{equation*}
\lim_{n\to\infty}\varepsilon_{a_n}^2\mu_i^{a_n}=\mu_i,\ \ i=1,\cdots,N.
\end{equation*}
By taking the weak limit of \eqref{lmw}, we thus obtain from \eqref{4-2} that $\hat w_i$ satisfies
\begin{equation*}
-\Delta  \hat w_i-\f53 a_N^*\rho_{\hat\g}^{\f23}\hat w_i=\mu_i \hat w_i\ \,\text{ in }\ \R^3,\ \,i=1,\cdots,N,
\end{equation*}
where $\rho_{\hat\g}=\sum_{i=1}^N|\hat w_i|^2$. Since Lemma \ref{4.1L} (1) gives that $\hat\g=\sum_{i=1}^N|\hat w_i\rangle\langle \hat w_i|$ is an optimizer of $a_N^*$, where $\hat w_i\in H^1(\R^3)$ satisfies $\langle \hat w_i,\hat w_j\rangle_{L^2(\R^3)}=\delta_{ij}$ for $i,\,j=1,\cdots,N$,  it follows from   \eqref{oH} that  $\mu_i$ is the $i$-th negative eigenvalue of the operator $H_{\hat\g}:=-\Delta-\frac{5a_N^*}{3}\rho_{\hat\g}^{\f23}$ in $L^2(\R^3)$, and hence $\mu_i<0$ for $i=1,\cdots,N$.

(3). We first claim that for $i=1,\cdots,N$,
\begin{equation}\label{uc0}
|\hat w_i^{a_n}|\leq C\ \ \hbox{and}\ \	\lim_{|x|\to\infty}|\hat w_i^{a_n}|\to0\ \  \text{uniformly for} \ \hbox{sufficiently large}\ \ n.
\end{equation}
Indeed, note from  \eqref{limmu} that $\mu_i^{a_n}<0$ as $n\to\infty$, $i=1,\cdots,N$. Applying Kato's inequality (cf. \cite[Theorem X.27]{Sim2}), it yields from \eqref{lmw} that
\begin{equation}\label{dgm}
\big[-\Delta+c_n(x)\big] |\hat w_i^{a_n}|\leq0\ \text{ in}\ \R^3\ \ \hbox{as}\ \ n\to\infty,\ \,i=1,\cdots, N,
\end{equation}
where $c_n(x):=-\f53 a_n\rho_{\hat{\g}_{a_n}}^{\f 23}<0$. Since $\big\{\hat w_i^{a_n}\big\}$ is bounded uniformly in $H^1(\R^3)$ as $n\to\infty$, by Sobolev's embedding theorem it gives that $\big\{\hat w_i^{a_n}\big\}$ is bounded uniformly in $L^q(\R^3)$  as $n\to\infty$, where $2\leq q\leq 6$. We therefore derive that $\big\{\rho_{\hat \g_{a_n}}^{\f23}\big\}$ is bounded uniformly in $L^r(\R^3)$ as $n\to\infty$, where $\f32\leq r\leq \f92$. By  De Giorgi-Nash-Moser theory (cf. \cite[Theorem 4.1]{hl}), we deduce from \eqref{dgm} that for $i=1,\cdots,N$,
\begin{equation}\label{uub}
\max_{B_1(\xi)} |\hat w_i^{a_n}|\leq C\Big(\int_{B_2(\xi)}|\hat w_i^{a_n}|^2dx\Big)^{\f12}\  \text{ uniformly for sufficiently large  }n>0,
\end{equation}
where $\xi\in\R^3$ is arbitrary, and  $C>0$ depends only on the bound of $\|c_n(x)\|_{L^2(B_2(\xi))}$. It then gives from \eqref{4-2} and \eqref{uub} that
$$|\hat w_i^{a_n}|\leq C\ \ \hbox{and}\ \ \lim_{|x|\to\infty}|\hat w_i^{a_n}|=0\ \ \hbox{uniformly as}\ \ n\to\infty,$$
and the claim \eqref{uc0} hence holds true.

We now prove that \eqref{wed} holds true. Applying  Kato's inequality (cf. \cite[Theorem X.27]{Sim2}) to \eqref{lmw} again,  we obtain from \eqref{limmu} and \eqref{uc0} that for sufficiently large $M>0$,
\begin{equation*}
\begin{split}
-\Delta|\hat w_i^{a_n}|&\leq-\varepsilon_{a_n}^2V(\varepsilon_{a_n}x + \hat x_{n})|\hat w_i^{a_n}|+\f53 a_n\rho_{\hat{\g}_{a_n}}^{\f 23}|\hat w_i^{a_n}|+\varepsilon_{a_n}^2\mu^{a_n}_i |\hat w_i^{a_n}|\\
&\leq \f12 \mu_i |\hat w_i^{a_n}|\ \ \text{ in }\  \R^3\backslash B_M(0)\ \text{ as }\ n\to\infty,\ \ i=1,\cdots,N.
\end{split}
\end{equation*}
By the comparison principle, it then follows from \eqref{uc0}  that
\begin{equation*}
|\hat w_i^{a_n}|\leq C\e^{-\sqrt{\f12|\mu_i|}|x| }\ \ \text{uniformly in}\ \, \R^3 \ \text{ as }\ n\to\infty,\,\ i=1,\cdots,N,
\end{equation*}
which further implies that
\begin{equation*}
\rho_{\hat\g_{a_n}}\leq Ce^{-\sqrt{2|\mu_N|}|x|}\ \ \hbox{uniformly in}\ \ \R^3\ \ \hbox{as}\ \ n\to\infty,
\end{equation*}
in view of the fact that $\mu_1<\mu_2\leq\cdots\leq\mu_N<0$. This therefore completes the proof of Lemma \ref{4.1L}.\qed
\vspace{5pt}

Using Lemma \ref{4.1L}, we next analyze the  following convergence of minimizers for $I_a(N)$ as $a\nearrow a_N^*$.
\begin{Lemma}\label{lem4.2}
Under the assumptions of Lemma \ref{4.1L}, assume $\big\{\hat w_i^{a_n}\big\}$ is the convergent subsequence obtained in Lemma \ref{4.1L}, where $\g_{a_n}=\dsum_{i=1}^N| u^{a_n}_i\rangle\langle u^{a_n}_i|$ is a minimizer of $I_{a_n}(N)$ satisfying $a_n\nearrow a_N^*$ as $n\to\infty$. Then up to a subsequence if necessary, for $i=1,\cdots,N$,
\begin{equation}\label{4-9}
w_i^{a_n}:=\varepsilon_{a_n}^{\f32}u_i^{a_n}(\varepsilon_{a_n} x+x_n)\to w_i\ \text{ strongly in }\  H^1(\R^3)\cap L^\infty(\R^3)\ \ \hbox{as}\ \ n\to\infty,
\end{equation}
  where  $\varepsilon_{a_n}=(a_N^*-a_n)^{\f14}>0$, and $\g:=\sum_{i=1}^N| w_i\rangle\langle  w_i|$  is an optimizer of $a_N^*$ defined by \eqref{a*}.  Moreover, the global maximum point $x_{n}\in\R^3$ of the density $\rho_{\g_{a_n}}:=\sum_{i=1}^N|u_i^{a_n}|^2$ satisfies
\begin{equation}\label{4-4}
\lim_{n\to\infty}x_n=x_0,
\end{equation}
where $x_0\in\R^3$ is a global minimum point of $V(x)$, i.e,  $V(x_0)=0$.
\end{Lemma}
	
\noindent{\bf Proof.}	
Note from \eqref{continuous} that $\rho_{\g_{a_n}}$ has at least one  maximum point $x_{n}\in\R^3$. Defining
\begin{equation}\label{hwi}
w_i^{a_n}:=\varepsilon_{a_n}^{\f32}u_i^{a_n}(\varepsilon_{a_n} x+ x_n),\ \ \widetilde\g_{a_n}:=\sum_{i=1}^N|w_i^{a_n}\rangle\langle w_i^{a_n}|,
\end{equation}
it then follows from \eqref{4-12}  that $w_i^{a_n}$ satisfies the following system
\begin{equation}\label{4-7}
-\Delta w_i^{a_n}+\varepsilon_{a_n}^2 V(\varepsilon_{a_n}x+x_n) w_i^{a_n} -\f53 a_n\rho_{\widetilde{\g}_{a_n}}^{\f 23} w_i^{a_n}=\varepsilon_{a_n}^2\mu^{a_n}_i w_i^{a_n}\ \,\text{ in }\  \R^3,\ \,i=1,\cdots,N,
\end{equation}
where $\rho_{\widetilde{\g}_{a_n}}:=\sum_{i=1}^N|w_i^{a_n}|^2$, and $\mu_i^{a_n}$ is the $i$-th eigenvalue of the operator $-\Delta+V(x)-\f{5a_n}{3}\rho_{\g_{a_n}}^{\f23}$ in $L^2(\R^3)$.

We first claim that
\begin{equation}\label{4-5}
\Big\{\f{\hat x_n-x_n}{\varepsilon_{a_n}}\Big\}\subset \R^3\ \ \hbox{is bounded uniformly in}\ \, n,
\end{equation}
where $\hat x_n\in\R^3$ is as in Lemma \ref{4.1L} (1). On the contrary, suppose that  $\l|\f{\hat x_n-x_n}{\varepsilon_{a_n}}\r|\to\infty$ as $n\to\infty$ after extracting  a subsequence. By the exponential decay \eqref{wed}, we deduce from \eqref{win} that
\begin{equation}\label{rlb}
\rho_{\g_{a_n}}(x_n)={\varepsilon_{a_n}^{-3}}\rho_{\hat\g_{a_n}}\Big(\f{x_n- \hat x_n}{\varepsilon_{a_n}}\Big)\leq C{\varepsilon_{a_n}^{-3}}\e ^{-\sqrt{2|\mu_N|}\f{|x_n- \hat x_n|}{\varepsilon_{a_n}}}=o(\varepsilon_{a_n}^{-3})\ \ \hbox{as}\ \ n\to\infty.
\end{equation}
On the other hand, direct calculations yield from \eqref{4-12} that
\begin{equation}\label{4-13}
\begin{split}
-\Delta \rho_{\g_{a_n}}&=-2\sum_{i=1}^N\Big(\big|\n u_i^{a_n}\big|^2+u_i^{a_n}\Delta u_i^{a_n}\Big)\\
&\leq2\sum_{i=1}^N\Big(-V(x)u_i^{a_n}+\f{5a_n}3\rho_{\g_{a_n}}^{\f23}u_i^{a_n}+\mu_i^{a_n} u^{a_n}_i\Big)u_i^{a_n}\\
&\leq \f{10a_n}3\rho_{\g_{a_n}}^{\f53}+2\mu_N^{a_n}\rho_{\g_{a_n}}\ \ \hbox{in}\ \ \R^3,
\end{split}
\end{equation}
due to the facts that $V(x)\geq0$   and  $\mu_1^{a_n}<\mu_2^{a_n}\leq\dots\leq\mu_N^{a_n}$. Since $x_n\in\R^3$ is the maximum point of $\rho_{\g_{a_n}}$ in $\R^3$, we further deduce from \eqref{limmu} and \eqref{4-13} that
\begin{equation*}
\rho_{\g_{a_n}}(x_n)=\sum_{i=1}^N| u^{a_n}_i(x_n)|^2\geq \Big(-\f{3\mu_N^{a_n}}{5a_n}\Big)^{\f 32}\geq C\varepsilon_{a_n}^{-3}\ \ \hbox{as}\ \ n\to\infty,
\end{equation*}
which however contradicts with \eqref{rlb}, and hence the claim \eqref{4-5} holds true.

It follows from \eqref{zA} and \eqref{4-5} that
$$\lim_{n\to\infty}x_n=\lim_{n\to\infty}\hat x_n=x_0,$$
where $x_0\in\R^3$ is a global minimum point of $V(x)$. This gives that \eqref{4-4} holds true. Moreover, we  obtain from \eqref{4-5} that  there exist a subsequence, still denoted by $\{a_n\}$, of $\{a_n\}$ and a point $x_*\in\R^3$ such that
$$\frac{ x_n-\hat x_n}{\varepsilon_{a_n}}\to x_*\ \ \hbox{as}\ \ n\to\infty.$$
One can then derive from \eqref{win}, \eqref{4-2} and \eqref{hwi} that for $i=1,\cdots,N$,
\begin{equation}\label{4-6}
\begin{aligned}
w_i^{a_n}(x):&=\varepsilon_{a_n}^{\f32}u_i^{a_n}(\varepsilon_{a_n} x+ x_n)=\hat w_i^{a_n}\Big(x+\frac{x_n-\hat x_n}{\varepsilon_{a_n}}\Big)\\
&\to\hat w_i(x+x_*):=w_i(x)\ \ \hbox{strongly in}\ \ H^1(\R^3)\ \ \hbox{as}\ \, n\to\infty.
\end{aligned}
\end{equation}
By the translation invariance of the minimization problem \eqref{a*}, we obtain that $\g:=\sum_{i=1}^N|w_i\rangle\langle w_i|$ is also an optimizer of $a_N^*$. Moreover, $\mu_i<0$ in \eqref{limmu} is also the $i$-th eigenvalue of the operator $H_\g:=-\Delta-\frac{5a_N^*}{3}\rho_\g^{\f23}$ in $L^2(\R^3)$.
		
In view of \eqref{4-6}, it remains to prove the   $L^\infty$-uniform convergence of $\big\{w_i^{a_n}\big\}$ as $n\to\infty$. The similar argument of  Lemma \ref{4.1L} (3) gives that for $i=1,\cdots,N$,
\begin{equation}\label{hed}
\big| w_i^{a_n}\big|\leq C\e^{-\sqrt{\f12|\mu_i|}|x| },\ \ \hbox{and}\ \ \rho_{\widetilde{\g}_{a_n}}\leq C e^{-\sqrt{2|\mu_N|}|x|}\ \ \text{uniformly in }\ \R^3\ \text{ as }\ n\to\infty,
\end{equation}
where $\mu_i<0$ is the $i$-th eigenvalue of the operator $H_\g$ in $L^2(\R^3)$. Defining
$$G_{i}^{a_n}(x):=-\varepsilon_{a_n}^2V(\varepsilon_{a_n}x+x_n)w_i^{a_n}+\frac{5a_n}{3}\rho_{\widetilde{\g}_{a_n}}^{\f23}w_i^{a_n}+\varepsilon_{a_n}^2\mu_i^{a_n}w_i^{a_n},\ \ i=1,\cdots,N,$$
so that  \eqref{4-7} can be rewritten as
\begin{equation}\label{4-8}
-\Delta w_i^{a_n}(x)=G_i^{a_n}(x)\ \,\ \hbox{in}\ \ \R^3,\ \ i=1,\cdots,N.
\end{equation}
Since  $\big\{w_i^{a_n}\big\}$ is bounded uniformly in $L^\infty(\R^3)$ as $n\to\infty$, one can obtain from \eqref{eV}, \eqref{limmu} and \eqref{4-4} that $\big\{G_i^{a_n}\big\}$ is bounded uniformly in $L^2_{loc}(\R^3)$ as $n\to\infty$. Applying  $L^p$ theory of \cite[Theorem 8.8]{gt}, it yields from \eqref{4-8} that for sufficiently large $R>0$,
\begin{equation*}
\| w_i^{a_n}\|_{H^{2}(B_R)}\leq C\l(\|  w_i^{a_n}\|_{H^1(B_{R+1})}+\|G_i^{a_n}\|_{L^2(B_{R+1})}\r),\,\,i=1,\cdots,N,
\end{equation*}
where $C>0$ is independent of $n>0$ and $R>0$. Hence,  $\big\{ w_i^{a_n}\big\}$ is bounded uniformly in $H^{2}(B_R)$ as $n\to\infty$ for $i=1,\cdots,N$. By the compact embedding theorem (cf. \cite[Theorem 7.26]{gt}), we thus conclude that there exist a subsequence, still denoted by $\big\{ w_i^{a_n}\big\}$, of $\big\{ w_i^{a_n}\big\}$ and a function $\widetilde{w}_i$  such that for above large $R>0$,
\begin{equation*}
w_i^{a_n}(x)\to \widetilde{w}_i(x) \ \text{ strong in }\ L^\infty(B_R)\ \text{ as }\ n\to\infty,\ \,i=1,\cdots,N.
\end{equation*}
We then derive from \eqref{4-6} that
\begin{equation}\label{lic}
w_i^{a_n}(x)\to  w_i(x) \ \text{ strong in }\ L_{loc}^\infty(\R^3)\ \text{ as }\ n\to\infty,\ \,i=1,\cdots,N.
\end{equation}
Since $\g=\dsum_{i=1}^N| w_i\rangle\langle w_i|$ is an optimizer of $a_N^*$, we obtain  from \eqref{Q} and \eqref{hed} that for any $\epsilon >0$, there exists a sufficiently large constant $M:=M(\epsilon)>0$, independent of $n>0$, such that for  sufficiently large $n>0$,
\begin{equation*}
|w_i^{a_n}(x)|,\ \ | w_i(x)|\leq\f\epsilon4\ \text{ in }\ \R^3\backslash B_M,\ \,i=1,\cdots,N,
\end{equation*}
which implies that
\begin{equation*}
\sup_{|x|\geq M}|\hat w_i^{a_n}-w_i(x)|\leq\f\epsilon2,\ \,i=1,\cdots,N.
\end{equation*}
We then obtain from  \eqref{lic} that
\begin{equation*}
w_i^{a_n}:=\varepsilon_{a_n}^{\f32}u_i^{a_n}(\varepsilon_{a_n} x+x_n)\to w_i\ \text{ strongly in }  L^\infty(\R^3)\ \text{ as }\ n\to\infty,\ \,i=1,\cdots,N.
\end{equation*}
This completes the proof of Lemma \ref{lem4.2}.\qed
\vspace{5pt}

Applying Lemma \ref{lem4.2}, we now complete the proof of  Theorem \ref{1.2T}.
\vspace{5pt}

\noindent{\bf Proof of Theorem \ref{1.2T}.}
Following Lemma \ref{lem4.2}, it suffices to prove \eqref{C_0} and \eqref{1-1}. We first note  from \eqref{a*} and  \eqref{hwi} that
\begin{equation}\label{an}
\begin{split}
I_{a_n}(N)=E_{a_n}(\gamma_{a_n})=&\f{1}{\varepsilon_{a_n}^2}\Big(\operatorname{Tr}(-\Delta\widetilde\gamma_{a_n})-a_N^*\int_{\R^3}\rho_{\widetilde\g_{a_n}}^{\f53}dx\Big)\\
&+\varepsilon_{a_n}^2\int_{\R^3}\rho_{\widetilde\g_{a_n}}^{\f53}dx+\int_{\R^3}V\l(\varepsilon_{a_n} x+x_n\r)\rho_{\widetilde\g_{a_n}}dx\\
\geq&\varepsilon_{a_n}^2\int_{\R^3}\rho_{\widetilde\g_{a_n}}^{\f53}dx+\int_{\R^3}V\l(\varepsilon_{a_n} x+x_n\r)\rho_{\widetilde\g_{a_n}}dx,
\end{split}
\end{equation}
where $\varepsilon_{a_n}=(a_N^*-a_n)^{\f14}>0$, $x_n$ is a global maximum point of $\rho_{\g_{a_n}}$, and $\lim_{n\to\infty}x_n=x_0\in\R^3$ satisfies $V(x_0)=0$.

Denote $(p_n,z_n):= x_n$ and $(p_0,0):=x_0$, where $p_n,\,p_0\in\R^2$ and $z_n\in\R$. We claim that $\Big\{\f{|p_n |-|p_0|}{\varepsilon_{a_n}}\Big\}\subset \R$ and $\Big\{\f{|z_n| }{\varepsilon_{a_n}}\Big\}\subset \R$ are both bounded uniformly in $n$. On the contrary, suppose that  up to a subsequence, one of them tends to infinity as $n\to\infty$. We then derive from \eqref{eV} and \eqref{4-9} that
\begin{equation}\label{4-3}
\begin{split}
&\lim_{n\to\infty}\varepsilon_{a_n}^{-2}\int_{\R^3}V\l(\varepsilon_{a_n} x+x_n\r)\rho_{\widetilde\g_{a_n}}dx\\
=&\lim_{n\to\infty}\int_{\R^3}\Big[\omega_1\Big(\Big|{p}+\f{p_n}{\varepsilon_{a_n}}\Big|-\f{|p_0|}{\varepsilon_{a_n}}\Big)^2+\omega_2\Big(x_3+\f{z_n}{\varepsilon_{a_n}}\Big)^2\Big]\rho_{\widetilde\g_{a_n}}dx=\infty,
\end{split}
\end{equation}
where we denote $p:=(x_1,x_2)\in\R^2$. We then get from  \eqref{an} and \eqref{4-3} that for any constant $C>0$,
\begin{equation*}
I_{a_n}(N)\geq C\varepsilon_{a_n}^2\ \ \hbox{as}\ \ n\to\infty,
\end{equation*}
which however contradicts with Theorem \ref{ee}. Hence, the above claim is true. We then obtain that there exists a subsequence, still denoted by $\{a_n\}$, of $\{a_n\}$ such that
\begin{equation*}
\lim_{n\to\infty}\f{|p_n|-|p_0|}{\varepsilon_{a_n}}=C_0\ \text{ and }\ \lim_{n\to\infty}\f{z_n }{\varepsilon_{a_n}}=C_1
\end{equation*}
hold for some constants $C_0$ and $C_1$. This therefore implies that \eqref{C_0} holds true.

We next prove   \eqref{1-1} as follows. Applying the dominated convergence theorem, it follows from \eqref{hed} that
\begin{equation}\label{4-11}
\begin{split}
&\quad\lim_{n\to\infty}\varepsilon_{a_n}^{-2}\int_{\R^3}V\l(\varepsilon_{a_n}x+x_n\r)\rho_{\widetilde\g_{a_n}}dx\\
&= \lim_{n\to\infty}\int_{\R^3}\Big[\omega_1\Big(\Big|{p}+\f{p_n}{\varepsilon_{a_n}}\Big|-\f{|p_0|}{\varepsilon_{a_n}}\Big)^2+\omega_2\Big(x_3+\f{z_n}{\varepsilon_{a_n}}\Big)^2\Big)\Big]\rho_{\widetilde\g_{a_n}}dx\\
&=\lim_{n\to\infty}\int_{\R^3}\Big[\omega_1\Big(\f{\l|\varepsilon_{a_n}p+p_n\r|-|p_n|}{\varepsilon_{a_n}}+\f{|p_n|-|p_0|}{\varepsilon_{a_n}}\Big)^2+\omega_2\Big(x_3+\f{z_n}{\varepsilon_{a_n}}\Big)^2\Big]\rho_{\widetilde\g_{a_n}}dx\\
&=\int_{\R^3}\Big[\omega_1\Big(\f{p\cdot p_0}{|p_0|}+C_0\Big)^2+\omega_2(x_3+C_1)^2\Big]\rho_{\g}dx\\
&=\int_{\R^3}\Big[\omega_1\Big(\f{(x_1,x_2)\cdot p_0}{A}+C_0\Big)^2+\omega_2(x_3+C_1)^2\Big]\rho_{\g}dx,
\end{split}
\end{equation}
where $\rho_{\g}=\sum_{i=1}^N|  w_i|^2$ is as in Lemma \ref{lem4.2}. We thus obtain from \eqref{4-9}, \eqref{an} and \eqref{4-11} that
\begin{equation*}
\lim_{n\to\infty}\varepsilon_{a_n}^{-2}{I_{a_n}(N)}\geq \int_{\R^3}\rho_{\g}^{\f53}dx+\int_{\R^3}\Big[\omega_1\Big(\f{(x_1,x_2)\cdot p_0}{A}+C_0\Big)^2+\omega_2(x_3+C_1)^2\Big]\rho_{\g}dx.
\end{equation*}
		
On the other hand, define
\begin{equation*}
u_i(x)=\varepsilon_{a_n}^{-\f 32} w_i\Big(\f {x-x_n}{\varepsilon_{a_n}}\Big),\ \ i=1,\cdots,N,\ \ \g':=\sum_{i=1}^N| u_i\rangle\langle u_i|,
\end{equation*}
where $w_i$ is as in \eqref{4-9}. Direct calculations then yield   that
\begin{equation*}
\begin{split}
I_{a_n}(N)&\leq \operatorname{Tr}(-\Delta+V)\g'-a_n\int_{\R^3}\rho_{\g'}^{\f 53}dx\\
&=\varepsilon_{a_n}^{-2}\Big[\operatorname{Tr}(-\Delta\g)-a_N^*\int_{\R^3}\rho_{\g}^{\f 53}dx\Big]+\varepsilon_{a_n}^2\int_{\R^3}\rho_{\g}^{\f 53}dx\\
&\quad+\varepsilon_{a_n}^2\int_{\R^3}\Big[\omega_1\Big(\Big|{p}+\f{p_n }{\varepsilon_{a_n}}\Big|-\f{|p_0|}{\varepsilon_{a_n}}\Big)^2+\omega_2\Big(x_3+\f{z_n}{\varepsilon_{a_n}}\Big)^2\Big]\rho_{\g}dx\\
&=\varepsilon_{a_n}^2\Big\{\int_{\R^3}\rho_{\g}^{\f 53}dx+\int_{\R^3}\Big[\omega_1\Big(\f{\l|\varepsilon_{a_n}p+p_n\r|-|p_n|}{\varepsilon_{a_n}}+\f{|p_n|-|p_0|}{\varepsilon_{a_n}}\Big)^2\\
&\qquad\qquad+\omega_2\Big(x_3+\f{z_n}{\varepsilon_{a_n}}\Big)^2\Big]\rho_{\g}dx\Big\}.
\end{split}
\end{equation*}
By the exponential decay of $\rho_{\g}$, we have
\begin{equation*}
\begin{split}
\lim_{n\to\infty}\varepsilon_{a_n}^{-2}I_{a_n}(N)&\leq\lim_{n\to\infty}\Big\{\int_{\R^3}\rho_{\g}^{\f 53}dx+\int_{\R^3}\Big[\omega_1\Big(\f{\l|\varepsilon_{a_n}p+p_n\r| -|p_n|}{\varepsilon_{a_n}}+\f{|p_n|-|p_0|}{\varepsilon_{a_n}}\Big)^2\\
&\qquad\qquad+\omega_2\Big(x_3+\f{z_n}{\varepsilon_{a_n}}\Big)^2\Big]\rho_{\g}dx\Big\}\\
&=\int_{\R^3}\rho_{\g}^{\f 53}dx+\int_{\R^3}\Big[\omega_1\Big(\f{(x_1,x_2)\cdot p_0}{A}+C_0\Big)^2+\omega_2(x_3+C_1)^2\Big]\rho_{\g}dx.
\end{split}
\end{equation*}
We thus conclude that
\begin{equation*}
\begin{split}
\lim_{n\to\infty}\varepsilon_{a_n}^{-2}I_{a_n}(N)=\int_{\R^3}\rho_{\g}^{\f 53}dx+\int_{\R^3}\Big[\omega_1\Big(\f{(x_1,x_2)\cdot p_0}{A}+C_0\Big)^2+\omega_2(x_3+C_1)^2\Big]\rho_{\g}dx,
\end{split}
\end{equation*}
and hence \eqref{1-1} holds true. This completes the proof of Theorem \ref{1.2T}.\qed

\vskip 0.3truein

\noindent {\bf Acknowledgements:} The authors are very grateful to the referees for their valuable suggestions and comments, which lead to to the great improvements of the present paper.

\end{document}